\documentclass{article}
\usepackage{amsfonts}
\usepackage{amsmath}
\usepackage{amssymb}
\usepackage{amsthm}
\usepackage{mathrsfs}
\usepackage{bm}
\usepackage{enumerate}
\usepackage{hyperref}
\usepackage{authblk}
\usepackage{framed,xcolor}
\usepackage[all]{xy}
\newcommand{\ud}{\mathrm{d}}
\newcommand{\dev}{\partial}
\newcommand{\R}{\mathbb{R}}
\newcommand{\Z}{\mathbb{Z}}
\newcommand{\A}{\mathcal{A}}
\newcommand{\M}{\mathcal{M}}
\newcommand{\X}{\mathfrak{X}}
\newcommand{\F}{\mathcal{F}}

\newcommand{\mdev}{\boldsymbol{\dev}}
\newcommand{\mlambda}{\boldsymbol{\lambda}}
\newcommand{\mmu}{\boldsymbol{\mu}}
\newcommand{\pc}{\circlearrowleft}
\renewcommand{\L}{\mathcal{L}}
\newcommand{\N}{\mathscr{N}}

\renewcommand{\vec}[1]{\mathbf{#1}}
\numberwithin{equation}{section}
\newtheorem{thm}{Theorem}
\newtheorem{lemma}{Lemma}
\theoremstyle{definition}
\newtheorem{deff}{Definition}

\title{On deformations of multidimensional Poisson brackets of hydrodynamic type}
\author{M.~Casati}
\affil{Scuola Internazionale Superiore di Studi Avanzati\\via Bonomea 265, 34136 Trieste (ITALY)}

\begin{document}
\maketitle
\begin{abstract}
 The theory of Poisson Vertex Algebras (PVAs) \cite{BdSK09} is a good framework to treat Hamiltonian partial differential equations. A PVA consists of a pair $(\A,\{\cdot_\lambda\cdot\})$ of a differential algebra $\A$ and a bilinear operation called the $\lambda$-bracket. We extend the definition to the class of algebras $\A$ endowed with $d\geq1$ commuting derivations. We call this structure a \emph{multidimensional PVA}: it is a suitable setting to study Hamiltonian PDEs with $d$ spatial dimensions. We apply this theory to the study of symmetries and deformations of the Poisson brackets of hydrodynamic type for $d=2$.
\end{abstract}
\section{Introduction}\label{sec:intro}
The main goal of this paper is to extend the formalism of Poisson Vertex Algebras (PVAs) in order to study Hamiltonian evolutionary PDEs with several spatial dimensions. Such PVAs provide efficient techniques to characterize Hamiltonian operators of evolutionary PDEs. In particular, we focus on multidimensional Poisson brackets of hydrodynamic type and study their symmetries and deformations, which correspond to the 1- and 2-cocycles in the Poisson--Lichnerowicz complex \cite{L77,dGMS05}. Let us first briefly present some basic facts in the theory of Poisson brackets of hydrodynamic type.
\subsection{Poisson brackets of hydrodynamic type and their deformations}
The notion of Poisson bracket of hydrodynamic type has been introduced in the early 1980s by Dubrovin and Novikov \cite{DN83} in order to characterize the Hamiltonian structure of a class of equations that describe systems such as ideal fluids with internal degrees of freedom. Poisson brackets (also called Poisson structures or Hamiltonian operators by different authors) of this form can be used to describe, for instance, Euler's equation \cite{N82}. The Hamiltonian structure is, in these cases, given in term of a first order differential operator
\begin{equation}\label{eq:hyPB_intro}
 P^{ij}(u(x))=g^{ij}(u(x))\frac{\ud}{\ud x}+b^{ij}_k(u)u^k_x
\end{equation}
for the maps $u^i(x)\in\M\colon S^1\to M$. $P$ is a well defined skewsymmetric operator if and only if the functions $g^{ij}(u)$ and $b^{ij}_k(u)$ are respectively the components  of a (pseudo)Riemannian contravariant metric and the contravariant Christoffel symbols of a compatible connection. Finally, $P$ is an Hamiltonian operator if and only if the connection has no torsion and the metric is flat. This means that there exists a coordinate system $\{u^i\}$ in which the Poisson structure is constant, the so-called \emph{Darboux coordinates}. For the 1-dimensional case, they are obviously the flat coordinates of the contravariant metric $g$.

Hydrodynamic Poisson brackets for maps $u(\vec{x})$, $\vec{x}=\{x^1,\ldots,x^d\}$ have been studied for several years \cite{DN84,M88}. They are defined by a family of $d$ maps $g^{ij\alpha}(u(\vec{x}))$ and $b^{ij\alpha}_k(u(\vec{x})$; in term of a differential operator, they have form
\begin{equation}\label{eq:mHYPB_intro}
P^{ij}(u(\vec{x}))=g^{ij\alpha}(u(\vec{x}))\frac{\ud}{\ud x^\alpha}+b^{ij\alpha}_k(u(\vec{x}))\left(\dev_\alpha u^k \right)
\end{equation}
In order for $P^{ij}$ to be a Hamiltonian operator, the functions $g^{ij\alpha}$ must be, for each $\alpha=1,\ldots d$, the components of a flat contravariant metric, while $b^{ij\alpha}_k$ the corresponding Christoffel symbols. These functions, moreover, must satisfy quite a large set of compatibility conditions, that we present in Theorem \ref{thm:Mokhov} \cite{M88}.
The main difference between the multidimensional brackets of hydrodynamic type and the monodimensional ones is that, despite each metric is flat, there do not necessarily exist Darboux coordinates for them: each metric can be constant, in general, for a different system of coordinates $\{u^i(\vec{x}z)\}_{i=1}^n$.

Under certain nondegeneracy assumptions there exist coordinate systems in which the multidimensional Poisson brackets of hydrodynamic type are at most linear \cite{DN84}. The Lie--Poisson bracket, namely the one associated to the Lie algebra of vector fields on a compact manifold \cite{N82}, is linear. It has the form
\begin{equation}\label{eq:LPhPB_intro}
 P_{ij}(p(\vec{x}))=p_i(\vec{x})\frac{\dev}{\dev x^j}+p_j(\vec{x})\frac{\dev}{\dev x^i}+\frac{\dev p_j(\vec{x})}{\dev x^i}
\end{equation}
for $p\colon M\to\Omega^1(M)$, $\dim M=d$. 

There exist Hamiltonian operators on the space $\M$ which are of order greater than one: one of the first examples to be discovered and probably the most celebrated one is the second Hamiltonian structure of KdV equation \cite{M78}
\begin{equation}\label{eq:Magri}
 P_2(u(x))=\frac{\ud^3}{\ud x^3}+4u\frac{\ud}{\ud x}+2u_x
\end{equation}
For finite-dimensional Poisson manifolds the celebrated Darboux theorem holds true, stating that for any Poisson bivector there locally exists a coordinate system in which it is constant and has a canonical form. The search for a canonical form of the Hamiltonian operators in the infinite dimensional manifold $M$ has been independently completed by several authors \cite{G01, dGMS05, DZ}. In \cite{DZ} and \cite{dGMS05}, this problem takes the name of \emph{triviality problem}, namely the problem of classificating all the deformations of a Poisson structures compatible with a chosed one of hydrodynamic type.

An Hamiltonian operator $P_0$ satisfies the Schouten relation $[P_0,P_0]=0$. We define a deformation $P_\epsilon$ of this operator as a formal infinite sum
\begin{equation*}
 P_\epsilon=P_0+\epsilon P_1+\epsilon^2 P_2+\cdots
\end{equation*}
such that the Schouten relation is satisfied:
\begin{equation*}
 [P_\epsilon,P_\epsilon]=0.
\end{equation*}
We say that such a deformation is trivial if there exists a general Miura type transform $\phi_\epsilon=\sum_{k=0}^\infty\epsilon^k\phi_k$ of the map space such that $\phi_{\epsilon*}P_0=P_\epsilon$ and $\phi_k$ is homogeneous of degree $k$. Getzler \cite{G01} proved that the infinite dimensional analogue of the Poisson--Lichnerowicz cohomology is trivial, hence in particular the second cohomology group (as directly proved also in \cite{dGMS05} and \cite{DZ}) vanishes. The vanishing of the second cohomology group implies that all the deformations of a given Poisson bracket of hydrodynamic type are trivial. Then, relying on Dubrovin and Novikov's result \cite{DN83}, for a certain coordinate system they must be trivial deformations of a \emph{constant} Poisson bracket. Conversely, there locally exists a system of coordinates in the space of maps for which the Poisson bracket is constant; such a system of coordinates can be regarded as the infinite dimensional analogue of the Darboux coordinates.
\paragraph*{Example}
For the case $n=1$, the constant Poisson structure of hydrodynamic type is the total derivative with respect to the spatial variable $x$. The second Hamiltonian structure of the KdV equation is easily got after a change of coordinates in the space $\M$
\begin{equation*}
 v=v(u)=u^2+iu'
\end{equation*}
where $\{u(x),u(y)\}=\delta'(x-y)$ and $u'(x)=u_x(x)$. Let us compute $\{v(x),v(y)\}$ with the help of the formula \eqref{eq:PBrackDens}
\begin{equation*}
\{f(u(x)),g(u(y))\}=\sum_{m,n\in\N}\frac{\dev f}{\dev u^{(m)}}\frac{\dev g}{\dev u^{(n)}}\dev_x^m\dev_y^n\{u(x),u(y)\} 
\end{equation*}
and recalling that $f(y)\delta^{(p)}(x-y)=\sum_{q=0}^pf^{(q)}(x)\delta^{(p-q)}(x-y)$ and $\dev_x\delta(x-y)=-\dev_y\delta(x-y)$. We get
\begin{equation*}
 \begin{split}
\{v(x),v(y)\}&=\{u^2(x)+iu'(x), u^2(y)+iu'\}\\
&=4u(x)u(y)\delta'(x-y)+2iu(y)\dev_x\delta'(x-y)+\\
&\quad+2iu(x)\dev_y\delta'(x-y)-\dev_x\dev_y\delta'(x-y)\\
&=\delta'''(x-y)+4(u^2(x)+iu'(x))\delta'(x-y)+\\
&\quad+2\dev_x(u^2(x)+iu'(x))\delta(x-y)\\
&=\left(\frac{\ud^3}{\ud x^3}+4v\frac{\ud}{\ud x}+2v'\right)\delta(x-y),
 \end{split}
\end{equation*}
namely the operator \eqref{eq:Magri} acting on the Dirac's delta.
\paragraph*{}
It must be stressed that this analogue of the Darboux theorem holds for the one-dimensional case. To our knowledge, there do not exist similar results even in dimension two. The research work by E.~Ferapontov and collaborators has produced a big outcome in the direction of classifying the integrable Hamiltonian equations of hydrodynamic type with $d$ spatial variables and their deformations (for instance, in \cite{FOS11,FNS12}); analogous results in the direction of deformations of the Poisson structure itself are not available. One of the main reason of this fact is that the required computations are very cumbersome, and their complexity increases dramatically with the order of the operators..
\subsection{Poisson Vertex Algebras}
The theory of Poisson Vertex Algebras \cite{BdSK09} provides a very effective framework to study Hamiltonian operators. The notion of a PVA, that can be seen as the semiclassical limit of Vertex Algebras \cite{K98}, has been introduced in order to deal with evolutionary Hamiltonian PDEs in which the unknown functions depend on one spatial variable. It provides a good framework for the study of integrability of such a class of equations, and also gives some insights into the study of nonlocal Poisson structures \cite{dSK13}. Let us first briefly introduce the notion of a (monodimensional) PVA.

A Poisson Vertex Algebra \cite{BdSK09} is a differential algebra $(\A,\dev)$ endowed with a bilinear operation $\A\times\A\to\R[\lambda]\otimes\A$ called a $\lambda$-bracket and satisfying the set of properties
\begin{enumerate}
\item $\{{f}_{\lambda}\dev g\}=(\lambda+\dev)\{f_{\lambda}g\}$
\item $\{\dev f_{\lambda}g\}=-\lambda\{f_{\lambda}g\}$
\item $\{f_{\lambda}gh\}=\{f_{\lambda}g\}h+\{f_{\lambda}h\}g$
\item $\{fg_{\lambda}h\}=\{f_{\lambda+\dev}h\}g+\{g_{\lambda+\dev}h\}f$
\item $\{g_{\lambda}f\}=-{}_{\to}\{f_{-\lambda-\dev}g\}$
\item $\{f_{\lambda}\{g_{\mu}h\}\}-\{g_{\mu}\{f_{\lambda}h\}\}=\{\{f_{\lambda}g\}_{\lambda+\mu}h\}$
\end{enumerate}
Let us explain the notation used in 4. and 5. Expand $\{f_{\lambda}g\}=\sum C_n\lambda^n$ with $C_n\in\A$. Then in each term of the RHS of equation 4 one has
\begin{equation*}
\{f_{\lambda+\dev}g\}h:=\sum_n C_n\left(\lambda+\dev\right)^nh 
\end{equation*}
Notice that using this convention $\{f_{\lambda+\dev}g\}=\{f_{\lambda+\dev}g\}1=\{f_{\lambda}g\}$. The RHS of the fifth equation is defined by
\begin{equation*}
{}_\to\{f_{-\lambda-\dev}g\}:=\sum(-\lambda-\dev)^nC_n
\end{equation*}

The main theorem, on which all the theory of PVA in the framework of Hamiltonian PDEs is based, is that from a $\lambda$-bracket of a PVA we can get the Poisson bracket between local functionals as
\begin{equation*}
 \left\{\int f,\int g\right\}=\int \{f_\lambda g\}\big|_{\lambda=0}.
\end{equation*}
Conversely, given a Poisson structure as a differential operator we can define a $\lambda$-bracket between the generators of a suitable differential algebra as the symbol of the differential operator; its extension to the full algebra is directly achieved by using the so called \emph{master formula}.

In the original paper and even in the more recent literature \cite{dSKV13} the theory of PVAs has been devoloped only for one dimensional Hamiltonian operators (in the original language, for a differential algebra with one derivation); since we want to deal with higher dimensional operators, we extend the definitions and the main theorems of \cite{BdSK09} introducing so-called \emph{multidimensional Poisson Vertex Algebras}, where the algebra $\A$ is endowed with $d$ commuting derivations. For a suitable $\A$, modelled on the algebra of differential polynomials of several variables, we show that the same axioms of a standard PVA, conveniently rephrased, can be used to characterize Poisson structures on this more general space of maps. In the paper, we apply the formalism of multidimensional PVAs to the $d$-dimensional Poisson brackets of hydrodynamic type. As an illustrative example, we obtain a new derivation of the set of necessary and sufficient conditions for a homogeneous differential operator of order 1 to be an Hamiltonian structure described by Mokhov in the late '80s \cite{M88}. Moreover, we apply the technique of multidimensional PVAs to the study of symmetries and deformations of such structures.

\begin{deff}[Multidimensional PVAs]
 A $d$-dimensional PVA is a differential algebra $\A$ endowed with $d$ commuting derivation $\dev_\alpha$, $\alpha=1,\ldots,d$ and with a bilinear operation $\{\cdot_{\mlambda}\cdot\}\colon\A\times\A\to\R[\lambda_1,\ldots,\lambda_d]\otimes\A$ called a $\lambda$-bracket of rank $d$. The $\lambda$-bracket of a multidimensional PVA satisfies the following set of properties
\begin{enumerate}
 \item $\{\dev_\alpha f_{\mlambda}g\}=-\lambda_\alpha\{f_{\mlambda}g\}$
 \item $\{f_{\mlambda}\dev_\alpha g\}=\left(\lambda_\alpha+\dev_\alpha\right)\{f_{\mlambda}g\}$
\item $\{f_{\mlambda}gh\}=\{f_{\mlambda}g\}h+\{f_{\mlambda}h\}g$
\item $\{fg_{\mlambda}h\}=\{f_{\mlambda+\mdev}h\}g+\{f_{\mlambda+\mdev}g\}h$
\item $\{g_{\mlambda}f\}=-{}_\to\{f_{-\mlambda-\mdev}g\}$
\item $\{f_{\mlambda}\{g_{\mmu}h\}\}-\{g_{\mmu}\{f_{\mlambda}h\}\}=\{\{f_{\mlambda}g\}_{\mlambda+\mmu}h\}$
\end{enumerate}
\end{deff}
The $\lambda$-bracket between two elements of $\A$ can be expanded as
\begin{equation*}
\{f_{\mlambda}g\}=\sum_{n_1,\ldots,n_d\in\Z_{\geq 0}} C_{n_1\ldots n_d}\lambda_1^{n_1}\cdots\lambda_d^{n_d}:=\sum_{N\in \Z^d_{\geq 0}}C_N\mlambda^N
\end{equation*}
with $C_N\in\A$. According to this decomposition, the RHS of the fourth property expands to
\begin{equation*}
\{f_{\mlambda+\mdev}g\}h=\sum_{N\in \Z^d_{\geq 0}}C_N(\mlambda+\mdev)^Nh:=\sum_{n_1,\ldots, n_d\in\Z_{\geq 0}}C_{n_1\ldots n_d}\left(\prod_{\alpha=1}^d(\lambda_\alpha+\dev_\alpha)^{n_\alpha}\right)h. 
\end{equation*}
The RHS of the fifth property is given by
\begin{equation*}
 {}_\to\{f_{-\mlambda-\mdev}g\}:=\sum_{N\in\Z^d_{\geq0}}\left(-\mlambda-\mdev\right)^NC_N=\sum_{n_1,\ldots,n_d\in\Z_{\geq 0}}\prod_{\alpha=1}^d(-\lambda_\alpha-\dev_\alpha)^{n_\alpha}C_{n_1\ldots n_d}.
\end{equation*}
For each PVA with $\A$ an algebra of differential polynomials, we can define a Lie bracket among the local functionals with density $\A$ (Theorem \ref{thm:mainPVA}). We have
\begin{equation}
 \left\{\int f,\int g\right\}:=\int\{f_{\mlambda}g\}\vert_{\mlambda=\boldsymbol{0}}.
\end{equation}
Such a Lie bracket is exactly what is called a Poisson structure. We prove that there is a one-to-one correspondence between Poisson structures on local functionals and $\lambda$-brackets on their densities.

Arnold described the relation between the Lie algebra of vector fields on a manifold and the Euler's equations of motion for incompressible fluids \cite{Ar66}, together with their Hamiltonian formulation. Novikov explicitly introduced a Poisson bracket for the system, which is the Lie--Poisson bracket of the algebra of vector fields \eqref{eq:LPhPB_intro} \cite{N82}. In Section \ref{sec:mHYPB} we consider the first order deformations of the Lie--Poisson structure \eqref{eq:LPhPB_intro}, whose equivalent $\lambda$-bracket is
\begin{equation*}
 \{{p_i}_{\mlambda}{p_j}\}=-\left(p_i\lambda_j+p_j\lambda_i+\dev_ip_j\right).
\end{equation*}
We give the set of necessary and sufficient conditions for a generic second order differential operator to be a deformation of the Lie--Poisson hydrodynamic bracket (Lemma \ref{thm:1stOrderDefoLP}).

It has been proved by Mokhov \cite{M08} that the Lie--Poisson hydrodynamic bracket is the normal form of the nondegenerate Poisson brackets of hydrodynamic type for $d=n=2$; Ferapontov and collaborators have completed the classification for the degenerate brackets in \cite{FOS11} and have recently provided the full classification for $d=2$ and $n\leq 4$ \cite{FLS13}. The main result in this paper is stated in Theorem \ref{thm:Trivia}. Any second order homogeneous differential operator compatible with an arbitrary Poisson bracket of hydrodynamic type for $d=n=2$ is a trivial deformation of the bracket itself.
\section{Multidimensional Poisson Vertex Algebras}\label{sec:PVA}
In this section we want to extend the notion of a Poisson Vertex Algebra \cite{BdSK09} in the spirit of the extension of the more general structure of a Lie conformal algebra into the category of Lie pseudoalgebras \cite{BdAK01}. Our ultimate aim is to identify Poisson brackets between local functionals on a space of maps with some $\lambda$-brackets, in the same way one can achieve this result for Poisson brackets on loop spaces.
\subsection{Formal map space}\label{ssec:mapspace}
Let $M$ be a $n$-dimensional smooth manifold. We want to descrive a class of Poisson brackets on the space \begin{equation*}
\M=\mathrm{Maps}(\Sigma\to M)
\end{equation*}
where $\Sigma$ is a compact $d$-dimensional smooth manifold. In order to avoid the problems arising from the integration, let us fix $\Sigma$ to be $(S^1)^d=T^d$. Anyway, all the definitions and the theorems of this section are expressed and work at the formal level. It means that it is not important whether we are considering functions on the $d$-torus, or rapidly decaying functions on $\R^d$; in a way, we are just working in the general space in which integration by parts is allowed producing no boundary terms.

We describe such a space of maps according to the theory of formal variational calculus \cite{GD75}. Our expositon is tightly related to the one of \cite{DZ}. Let us define the formal map space $\M$ in terms of ring of functions on it.

Let $U\subset M$ be a chart on $M$ with coordinates $(u^1,\ldots,u^n)$ and denote $\A=\A(U)$ the space of polynomials in the independent variables $u^i_I$ for $i=1,\ldots,n$ and $I\in\Z_+^d$ a multiindex (i.e., $I_\alpha=1,2,\ldots$ with $\alpha=1,\ldots,d$)
\begin{equation}\label{eq:difpoly}
 f(x,u;u_I):=\sum_{m\geq 0}f_{i_1I_1;\ldots;i_mI_m}(x,u)u^{i_1}_{I_1}\ldots u^{i_m}_{I_m}
\end{equation}
with coefficients $f_{i_1I_1;\ldots;i_mI_m}(\vec{x},u)$ smooth functions on $\Sigma\times M$. Such an expression is called a \emph{differential polynomial}. The space $\A$, endowed with a family of operators
\begin{align*}
\dev_\alpha\colon\A&\to\A\\
\;f&\mapsto\frac{\dev f}{\dev x^\alpha}+u^i_{E_\alpha}\frac{\dev f}{\dev u^i}+u^i_{I+E_\alpha}\frac{\dev f}{\dev u^i_I}
\end{align*}
($\alpha=1,\ldots,d$ and $E_\alpha=(0,0,\ldots,\underbrace{1}_\alpha,0,\ldots,0)$) satisfying the following commutation properties
\begin{subequations}\label{eq:derComm}
\begin{align}
 [\dev_\alpha,\dev_\beta]&=0&\forall\, \alpha,\beta\\\label{eq:commDeriv}
\left[\frac{\dev}{\dev u^i_I},\dev_\alpha\right]&=\frac{\dev}{\dev u^i_{I-E\alpha}}&(=0 \text{ if }I_\alpha=0)\\
\left[\frac{\dev}{\dev u^i_I},\frac{\dev}{\dev u^j_J}\right]&=0&\forall\, (i,j,I,J)
\end{align}
\end{subequations}
form what in \cite{BdSK09} is called an \emph{algebra of differential polynomials}.

Since we are interested in local (in the sense of \cite{DZ}) structures on the space of maps, we do not have to take into account the explicit dependence on the points in $\Sigma$. This justifies the following definitions, where we will restrict ourselves to consider the space $\hat{\A}\subset\A$ of differential polynomials $f$ that do not depend explicitly on $x^\alpha$. The `total derivatives' have thus the form
\begin{equation}\label{eq:totDer}
 \dev_\alpha=\sum_{\substack{i=1,\ldots,n\\I\in\Z^d_{\geq 0}}}u^i_{I+E_\alpha}\frac{\dev}{\dev u^i_I}
\end{equation}
and satisfy the same commutation relations as in \eqref{eq:derComm}.

Because of the lacking of dependence on the variables on $\Sigma$, we are allowed to identify the space of local functionals $\hat{\F}$ whose densities do not depend explicitly on $x$ with the quotient $\hat{\A}/\sum_\alpha\dev_\alpha\hat{\A}$. The quotient operation is denoted $\int$.

Let us consider the space $\X(\hat{\A})$ of vector fields on the formal space of maps. These are formal infinite sums
\begin{equation}\label{eq:Vfield}
 \xi=\sum_{I\in\Z^d_{\geq 0}}\xi^i_I(u_J)\frac{\dev}{\dev u^i_I}
\end{equation}
with $\xi^i_I\in\hat{\A}$.

The derivative of a local functional $\int f\equiv\overline{f}\in\hat{\F}$ along a vector field $\xi$ reads
\begin{equation}
 \xi\overline{f}=\int\sum_{I\in\Z^d_{\geq 0}}\xi^i_I(u_J)\frac{\dev f}{\dev u^i_I}\ud\vec{x}
\end{equation}
while the Lie bracket between two of such vector fields is obtained in a straightforward way by composition of derivations $[\xi,\eta]f=\xi(\eta f)-\eta(\xi f)$.
The total derivatives $\dev_\alpha$ can be regarded as vector fields with $\xi^i_I=u^i_{I+E_\alpha}$.

An evolutionary vector field is a derivation of $\hat{\A}$ which commutes with all the derivations $\dev_\alpha$. A simple computation shows that the condition imposes $\xi^i_I=\dev_\alpha\xi^i_{I-E_\alpha}$. Applying this relation reculsively we get that an evolutionary vector field has form
\begin{equation}\label{eq:evVF_PVA}
 \xi=\sum_{\substack{i=1,\ldots,n\\I\in\Z^d_{\geq 0}}}(\mdev^I X^i(u_J))\frac{\dev}{\dev u^i_I}.
\end{equation}
We will adopt a multi-index notation and denote
\begin{equation*}
 \prod_{\alpha=1}^d(\dev_\alpha)^{I_\alpha}=:\mdev^I.
\end{equation*}
for $I\in \Z^d_{\geq 0}$. Analogue conventions will be adopted also for more general operators or expressions, always meaning that they must be regarded as ``term by term'' powers.

$\hat{\A}^n$, as a collection of $n$ elements of $\hat{\A}$, can be regarded as a vector. We can introduce a symmetric bilinear pairing $\hat{\A}^n\times\hat{\A}^n\to\hat{\F}$ given by
\begin{equation}\label{eq:pair}
 (A,B)\mapsto\int\sum_{i=1}^n A_i\cdot B_i
\end{equation}
and use it to identify $\hat{\A}^n$ with its dual space, namely the one-forms. The \emph{variational derivative}, in this setting, is a map $\hat{\F}\to {\hat{\A}^n}{}^*$. We can write $\delta \overline{f}=\frac{\delta \overline{f}}{\delta u^i}\delta u^i$, and each component is 
\begin{equation}\label{eq:VarDer}
 \frac{\delta \overline{f}}{\delta u^i}:=\sum_I\left(-\mdev \right)^I\frac{\dev f}{\dev u^i_I}.
\end{equation}
It is worthy noticing that we are giving as definition a formula which can actually be regarded as a proposition following by the rigorous construction of a \emph{variational bicomplex}, which can be found for instance in \cite{A92}. 
\begin{lemma}
On elements of $\hat{\A}$, the variational derivative of a total derivative vanishes. Moreover, for any evolutionary vector field $\xi\in\X(\A)$  \eqref{eq:evVF_PVA} and for any $f\in\hat{\A}$, we have
\begin{equation}
 \int \xi f=\int X^i\frac{\delta f}{\delta u^i}.
\end{equation}
\end{lemma}
\begin{proof}
Applying the commutation rule \eqref{eq:commDeriv} to the definition we get 
\begin{equation}
  \frac{\delta}{\delta u^i}\dev_\alpha f=\sum_I\left((-1)^{|I|}\mdev^{I+E_\alpha}\frac{\dev f}{\dev u^i_I}+(-1)^{|I|}\mdev^I\frac{\dev f}{\dev u^i_{I-E_\alpha}}\right).
\end{equation}
Now it is sufficient to impose $I'=I-E_\alpha$ in the second term to get the result
\begin{equation}
\sum_{I,I'}(-1)^{|I|}\mdev^{I+E_\alpha}\frac{\dev f}{\dev u^i_I}-(-1)^{|I'|}\mdev^{I'+E_\alpha}\frac{\dev f}{\dev u^i_{I'}}=0.
\end{equation}
In order to prove the second part of the lemma it is enough to integrate by parts the definition \eqref{eq:evVF_PVA}.
\end{proof}
\subsection{Poisson bivector and Poisson bracket}\label{ssec:Poisson}
A bivector, in general, is an element of $\X(\hat{\A})^{\wedge 2}$. In components, bivectors are written as infinite sums of expression of the form
\begin{equation}\label{eq:biVfield}
\alpha=\frac{1}{2}\alpha^{i_1;i_2}_{I_1;I_2}\left(u(x_1),u(x_2);\ldots\right)\frac{\dev}{\dev u^{i_1}_{I_1}(x_1)}\wedge\frac{\dev}{\dev u^{i_2}_{I_2}(x_2)}
\end{equation}
antisymmetric with respect to simultaneous exchange of $i_1,I_1,x_1\leftrightarrow i_2,I_2,x_2$.
The Lie bracket between vector fields and the wedge product allow us to define the \emph{Schouten-Nijenhuis} bracket between $k$-vectors. It is defined by extending the Lie bracket of vector fields with respect to the product, imposing the Leibniz rule $[\alpha,\beta\wedge \gamma]=[\alpha,\beta]\wedge\gamma+(-1)^{(k-1)l}\beta\wedge[\alpha,\gamma]$ if $\alpha\in\X^k$ and $\beta\in\X^l$. It turns out that the Schouten-Nijenhuis bracket is the unique extension of Lie bracket of vector fields which turns the graded algebra of $k$-vector fields into a Gerstenhaber algebra. The formulas for the Schouten-Nijenhuis bracket between generic $k$- and $l$- vectors are in general quite involved. We give the formula for the bracket of a bivector with a vector fields, which is equivalent to the Lie derivative. Given a vector field $\xi=\xi^i_I(x,u(x),\ldots)\dev/\dev u^i_I$ and $\beta$ a bivector of form \eqref{eq:biVfield} we get
\begin{equation}\label{eq:LderKV}
 \L_\xi\beta^{i_1,i_2}_{I_1,I_2}=\sum_{l,L}\xi^{l}_{L}(u_l)\frac{\dev}{\dev u^{l}_{L}(x_l)}\beta^{i_1,i_2}_{I_1I_2}-\frac{\dev\xi^{i_1}_{I_1}}{\dev u^{l}_{L}(x_l)}\beta^{l,i_2}_{L,I_2}-\frac{\dev\xi^{i_2}_{I_2}}{\dev u^{l}_{L}(x_l)}\beta^{i_1,l}_{I_1,L}.
\end{equation}
This formula is useful to define \emph{translational invariant} $k$-vectors, which satisfy
\begin{equation}\label{eq:TransInvKV}
\L_{\dev_\alpha}\beta=0
\end{equation}
for all $\alpha=1,\ldots,d$. The set of conditions \eqref{eq:TransInvKV} imposes relations among the components of the bivector analogue to the ones we have already given for evolutionary vector fields. Every translational invariant bivector $\beta$ has coefficients
\begin{equation}\label{eq:TransInvKVcompo}
 \beta^{i_1,i_2}_{I_1I_2}(u_J(x_1),u_K(x_2))=\left(\mdev_{x_1}\right)^{I_1}\left(\mdev_{x_2}\right)^{I_2}B^{i_1i_2}(u_J(x_1),u_K(x_2))
\end{equation}
where the differential polynomials $B^{i_1i_2}$ are antisymmetric in the simultaneous exchange of indices.
We call such differential polynomials the \emph{components} of the translational invariant bivector $\beta$.

A \emph{local bivector} is a translational invariant bivector such that its dependence on $x_1,x_2$ is given by a distribution with the support on the diagonal $x_1=x_2$, i.e.
\begin{equation}\label{eq:Biv}
 B^{ij}=\sum_{P\in Z^d_{\geq0}}C^{ij}_{P}(u_J(x))\mdev _{x}^{P}\delta(x-y).
\end{equation}

The delta functions and their derivatives are defined by the usual formulae
\begin{gather*}
 \int f(y)\delta(x-y)\ud\vec{y}=f(x)\\
\int f(y)\mdev _x^I\delta(x-y)\ud\vec{y}=\int f(y)(-\mdev _y)^I\delta(x-y)\ud\vec{y}=\mdev ^If(x)\\
\begin{split}\int f(x_1,\ldots,x_k)\mdev _{x_1}^{I_2}\delta(x_1-x_2)\ldots\mdev _{x_1}^{I_k}\delta(x_1-x_k)\ud \vec{x}_2\ldots\ud\vec{x}_k=\\=\mdev _{x_2}^{I_2}\ldots\mdev _{x_k}^{I_k}f(x_1,\ldots,x_k)\big\vert_{x_1=x_2=\cdots=x_k}.\end{split}
\end{gather*}
Where we do not specify the variables on which the derivatives $\mdev $ act, it is meant that they act on the first ones.
The value of a local bivector on two one-forms $\phi=\phi_{i}(x,\ldots)\delta u^i$ and $\psi=\psi_{i}(y,\ldots)\delta u^i$ is
\begin{equation}
 \int\phi_i B^{ij}(x,\ldots)\mdev \psi_j\ud\vec{x}
\end{equation}
which easily follows from simply pairing the one-forms with the bivectors and integrating in such a way that the derivatives of Dirac's delta act on the second one form. The antisymmetry of the bivector, namely $(B\phi,\psi)=-(B\psi,\phi)$ imposes on the components
\begin{equation}\label{eq:antisBiv}
 B^{ji}_S=\sum_{T\in\Z^d_{\geq 0}}(-1)^{|T|+1}\binom{T}{S}\mdev ^{T-S}B^{ij}_T
\end{equation}
where we denote $|T|=\sum_\alpha T_\alpha$ and use the binomial coefficient for multi-indices
\begin{equation}\label{eq:binomMulti}
 \binom{A}{B}=\binom{a_1}{b_1}\cdots\binom{a_d}{b_d}
\end{equation}
and
\begin{equation}\label{eq:binom}
 \binom{a}{b}=\begin{cases}
               \frac{a!}{b!(a-b)!}&0\leq b\leq a\\
0 &\text{otherwise.}
              \end{cases}
\end{equation}
We will occasionally use also the multinomial coefficients
\begin{equation*}
 \binom{A}{B_1,\ldots, B_n},\qquad B_n=A-\textstyle{\sum_{i=1}^{n-1}}B_i
\end{equation*}
which definition is analogue to the one for binomial coefficients with multi-indices, given the usual multinomial coefficient
\begin{equation}\label{eq:multin}
 \binom{a}{b_1\ldots b_n}=\frac{a!}{b_1!b_2!\ldots b_n!}.
\end{equation}
From the componentwise expression \eqref{eq:Biv} we see that each component can be interpreted as a differential operator acting on the Dirac's delta
\begin{equation}
 B^{ij}(x,u(x),u(x)_I;\frac{\ud}{\ud x})\delta(x-y)
\end{equation}
with
\begin{equation*}
 B^{ij}(x,u(x),u(x)_I;\frac{\ud}{\ud x})=\sum B^{ij}_S\mdev ^S.
\end{equation*}

A \emph{local Poisson structure} is a local bivector $P\in\X^{\wedge 2}(\M)$ satisfying the Schouten relation $[P,P]=0$.

Given a Poisson bivector it is possible to define a bilinear operation (that we will call a bracket) on the space of local densities $f\in\A$. It can be used to define a bracket on the space $\hat{\F}$ of local functionals, which is usually called the \emph{Poisson bracket} of functionals. This name is somehow confusing since the Poisson bracket of functionals is not the bracket on a Poisson algebra; indeed, it fails to be a derivation, because of the lack of a product in the space of functionals.

Given a Poisson structure $P$, we first define the bracket on $\hat{\A}$ on the basis elements $u^i$; we will often refer to them as the \emph{generators} of $\hat{\A}$. We have
\begin{equation}\label{eq:PBrackGen}
 \{u^i(x),u^j(y)\}=\sum_S P^{ij}_S\left(u(x),u_I(x)\right)\mdev ^S\delta(x-y).
\end{equation}
 This definition extends to two generic densities $f,g\in\hat{\A}$ according to the formula
\begin{equation}\label{eq:PBrackDens}
 \{f(x),g(y)\}=\sum_{L,M}\frac{\dev f}{\dev u^i_L(x)}\frac{\dev g}{\dev u^j_M(y)}\mdev _x^L\mdev _y^M\{u^i(x),u^j(y)\}.
\end{equation}
Such a bracket satisfies by definition the Leibniz rule, i.e. $\{f,gh\}=\{f,g\}h+g\{f,h\}$ and it is obviously bilinear. An important remark is that such a bracket does not satisfy neither the usual skewsymmetry property nor the Jacobi identity, thus it is not a Lie bracket and, \emph{a fortiori}, not even a Poisson bracket. The reason why the two important properties do not hold is quite natural: we defined a Poisson bivector to be skewsymmetric in the sense \eqref{eq:antisBiv}, which means that the skewsymmetry makes sense only after the integration, i.e. on $\hat{\F}$. On $\hat{\F}$ Leibniz property does not hold, but we can give a genuine Lie bracket.
\begin{deff}[Poisson bracket]
 A Poisson bracket $\{,\}$ in $\hat{\F}=\hat{\A}/\ud\hat{\A}$ is a bilinear operation
\begin{align*}
 \{\cdot,\cdot\}\colon\hat{\F}\times\hat{\F}&\to\hat{\F}\\
\qquad\left(\int f,\int g\right)&\mapsto\left\{\int f,\int g\right\}
\end{align*}
satisfying the following two fundamental properties:
\begin{enumerate}
 \item Skewsymmetry: $\{\int f,\int g\}=-\{\int g,\int f\}$
\item Jacobi identity: $\{\int f,\{\int g,\int h\}\}-\{\int g,\{\int f,\int h\}\}=\{\{\int f,\int g\},\int h\}$
\end{enumerate}
Applying the skewsymmetry property, Jacobi identity can also be written as the vanishing of the expression $\{\{\int f,\int g\},\int h\}+\mathrm{cycl.}(f,g,h)=0$ which is the usual way to write it.

Given a Poisson bivector $P$ of form \eqref{eq:Biv} satisfying $[P,P]=0$, the Poisson bracket of two local functionals $\int f(u(x),u_I(x))\ud \vec{x}$ and $\int g(u(y),u_I(y))\ud \vec{y}$ is given by
\begin{equation}\label{eq:PBrackFunct}
\begin{split}
 \left\{\int f,\int g\right\}&=\int\int \left\{f(x),g(y)\right\}\ud \vec{x}\ud\vec{y}\\
&=\int\int\sum_{L,M}\frac{\dev f}{\dev u^i_L(x)}\frac{\dev g}{\dev u^j_M(y)}\mdev _x^L\mdev _y^M\{u^i(x),u^j(y)\}\ud\vec{x}\ud\vec{y}\\
&=\int\int\frac{\delta f}{\delta u^i(x)}\frac{\delta g}{\delta u^j(y)}\{u^i(x),u^j(y)\}\ud\vec{x}\ud\vec{y}\\
&=\sum_{S\in\Z^d_{\geq 0}}\int\frac{\delta f}{\delta u^i(x)} P^{ij}_S(u(x),u_I(x))\mdev ^S\frac{\delta g}{\delta u^j(x)}\ud\vec{x}
\end{split}
\end{equation}
where the second equality is given by \eqref{eq:PBrackDens}, the third one is obtained by integrating by parts and transferring the total derivatives $\mdev $ on the partial derivatives of $f$ and $g$ respectively and the fourth one by performing the integration with respect to $\vec{y}$ for the Dirac's delta.
\end{deff}
We do not prove here that the Schouten condition for $P$ is the crucial requirement for the bracket to be Lie, namely to satisfy Jacobi identity. Although it is possible to get this result with the Dirac's delta formalism -- or even by regarding the Poisson bracket of densities as a distribution and evaluating it on test functions -- we are going to shift our point of view and consider Barakat, De Sole and Kac's approach to Hamiltonian operators on a space of maps in terms of Poisson Vertex Algebras \cite{BdSK09}.

\subsection{Poisson Vertex Algebras}\label{ssec:mPVA}
Let $\hat{\A}$ be a differential algebra with $d$ commuting derivations. Usually, we consider the algebra of differential polynomials or an extension thereof.

\begin{deff}[$\lambda$-bracket]
 A $\lambda$-bracket (of rank $d$) on $\hat{\A}$ is a $\R$-linear map
\begin{align*}
 \{\cdot_{\mlambda}\cdot\}\colon\hat{\A}\times\hat{\A}&\to\R[\lambda_1,\ldots,\lambda_d]\otimes\hat{\A}\\
(f,g)&\mapsto\{f_{\mlambda} g\}
\end{align*}
which is \emph{sesquilinear}, namely
\begin{subequations}\label{eq:sesquilDef}
\begin{align}\label{eq:lsesqui}
 \{\dev_\alpha f_{\mlambda} g\}&=-\lambda_\alpha\{f_{\mlambda} g\}\\ \label{eq:rsesqui}
\{f_{\mlambda}\dev_\alpha g\}&=\left(\dev_\alpha+\lambda_\alpha\right)\{f_{\mlambda} g\}
\end{align}
\end{subequations}
and obeys, respectively, the \emph{right} and \emph{left Leibniz rule}
\begin{subequations}\label{eq:rlLeibDef}
 \begin{align}\label{eq:rLeib}
\{f_{\mlambda} gh\}&=\{f_{\mlambda} g\}h+\{f_{\mlambda} h\}g\\ \label{eq:lLeib}
\{fg_{\mlambda} h\}&=\{f_{\mlambda+\mdev}h\}g+\{g_{\mlambda+\mdev}h\}f
 \end{align}
\end{subequations}
\end{deff}
By definition, the $\lambda$-bracket of two elements in $\hat{\A}$ is a polynomial in $\lambda_1,\ldots,\lambda_d$ (we will often refer to the collection of $\lambda_\alpha$ as $\mlambda$) with coefficients in $\hat{\A}$. In general, we can write $\{f_{\mlambda} g\}=A(f,g)_{i_1,\ldots,i_d}\lambda_1^{i_1}\ldots\lambda_d^{i_d}$ which, using the usual multiindex notation, is equivalent to writing $A(f,g)_I\mlambda^I$. When, as in \eqref{eq:lLeib}, we write $\{f_{\mlambda+\mdev}g\}$ it means that the $\lambda$ product is $A(f,g)_I\left(\mlambda+\mdev\right)^I$, with the derivation acting on the right (if nothing is written on the right, it is equivalent to the derivatives acting on 1 and thus the only term not vanishing is $\mlambda^I$).
\begin{deff}[Multidimensional Poisson Vertex Algebra]
 A ($d$-dimensional) \emph{Poisson Vertex Algebra} is a differential algebra $\hat{\A}$ endowed with a $\lambda$-bracket of rank $d$ which is \emph{skewsymmetric}
\begin{equation}\label{eq:LambdaSkew}
 \{g_{\mlambda}f\}=-{}_{\to}\{f_{-\mlambda-\mdev}g\}
\end{equation}
and satisfy the \emph{PVA-Jacobi identity}
\begin{equation}\label{eq:LambdaJac}
 \{f_{\mlambda}\{g_{\mmu}h\}\}-\{g_{\mmu}\{f_{\mlambda}h\}\}=\{\{f_{\mlambda}g\}_{\mlambda+\mmu}h\}.
\end{equation}
\end{deff}
The notation used in \eqref{eq:LambdaSkew} means that the differential operators $(-\mlambda-\mdev)$ must be regarded as acting on the coefficient of the bracket, too; namely ${}_{\to}\{f_{-\mlambda-\mdev}g\}=(-\mlambda-\mdev)^IA(f,g)_I$.
\begin{thm}[Master formula]\label{thm:master} Let $\hat{\A}$ be the algebra of differential polynomials (or an extension thereof) as defined in Section \ref{ssec:mapspace}. Given two elements $(f,g)\in\hat{\A}$, their $\lambda$-bracket can be expressed in terms of the $\lambda$-bracket between the so-called \emph{generators} of $\hat{\A}$, $\{u^i\}_{i=1,\ldots,n}$. We have
 \begin{equation}\label{eq:MasterFormula}
  \{f_{\mlambda}g\}=\sum_{\substack{i,j=1\ldots,n\\M,N\in\Z^d_{\geq 0}}}\frac{\dev g}{\dev u^j_N}(\mlambda+\mdev)^N\{u^i_{\mlambda+\mdev}u^j\}(-\mlambda-\mdev)^M\frac{\dev f}{\dev u^i_M}.
 \end{equation}
In particular, the skewsymmetry and the PVA-Jacobi property hold if and only if the same properties for the generators hold.
\end{thm}
We give here only a sketch of the proof of the theorem. The complete -- rather cumbersome -- proof extends to the $d$-dimensional case the Theorem 1.15 of \cite{BdSK09} and follows the same ideas, without major technical issues. Our aim is to prove that the master formula provides the unique bilinear opeation satisfying the properties of a PVA for any two elements of $\hat{\A}$. From sesquilinearity of the bracket between two generators we have that $\{u^i_M {}_{\mlambda}u^j_N\}=(\mlambda+\mdev)^N(-\mlambda)^M\{u^i_{\mlambda}u^j\}$. Moreover, from the right Leibniz property \eqref{eq:rLeib} follows that $\{f_{\mlambda}\cdot\}$ is a derivation of $\hat{\A}$, thus it acts on $g$ only by its derivatives $\frac{\dev g}{\dev u^j_N}$. We get
\begin{equation}
 \{f_{\mlambda}g\}=\sum\{f_{\mlambda}u^j_N\}\frac{\dev g}{\dev u^j_N}.
\end{equation}
Applying the sesquilinearity \eqref{eq:rsesqui} we thus obtain
\begin{equation}\label{eq:MasterRight}
 \{f_{\mlambda}g\}=\sum\frac{\dev g}{\dev u^j_N}(\mlambda+\mdev)^N\{f_{\mlambda}u^j\}
\end{equation}
where the partial derivatives pf $g$ have been put on the left to denote that the total derivatives in the parenthesis act only on the $\lambda$-bracket itself, according to the right Leibniz rule \eqref{eq:rLeib}.

The way in which the derivatives of the first function $f$ enter into the master formula, conversely, is dictated by the left Leibniz rule and the sesquilinearity for the first entry of the bracket. We have
\begin{equation}\label{eq:MasterLeft}
 \{f_{\mlambda}g\}=\sum\{u^i_{\mlambda+\mdev}g\}(-\mlambda-\mdev)^M\frac{\dev f}{\dev u^i_M}.
\end{equation}
Note that in \eqref{eq:MasterLeft} the total derivatives act also on the partial derivatives of $f$, as imposed by the left Leibniz rule \eqref{eq:lLeib}. One can then prove that the skewsymmetry and the PVA-Jacobi identity for the brackets between the generators are the only conditions needed for the corresponding properties between generic elements of $\hat{\A}$. 

In Section \ref{ssec:Poisson} we have noticed that there is a remarkable difference between the bracket defined by the same Poisson bivector in the space of local densities and the one among local functionals. In short, while the former is not a Lie bracket but it is a derivation, the latter -- despite being an actual Lie bracket -- fails at being the bracket of a Poisson algebra. The main discovery which establishes a relation between the theory of Hamiltonian PDEs and Poisson Vertex Algebras has originally been proved in \cite{BdSK09} for a PVA of rank 1, namely that the Poisson bracket (strictly speaking, the bracket defined by a Poisson bivector) among local densities is related to a $\lambda$-bracket by the relation
\begin{equation}\label{eq:PoissonLambda}
 \{f,g\}=\{f_{\mlambda}g\}\big\vert_{\mlambda=0}\qquad f,g\in\hat{\A}.
\end{equation}
Its extension to the more general case we are dealing with is straightforward. This fact is summarized by the following
\begin{thm}\label{thm:mainPVA}
 Let $\hat{\A}$ be an algebra of differential polynomials with a $\lambda$-bracket and consider the bracket on $\hat{\A}$ defined in \eqref{eq:PoissonLambda}. Then
\begin{enumerate}[(a)]
\item The bracket \eqref{eq:PoissonLambda} induces a well-defined bracket on the quotient space $\hat{\F}$;
\item If the $\lambda$-bracket satisfies the axioms of a PVA, then the induced bracket on $\hat{\F}$ is a Lie bracket.
\end{enumerate}
\end{thm}
\begin{proof}
Part (a). From the property of sesquilinearity we have that, for any $\alpha=1,\ldots,d$,
\begin{align}
 \{f+\mdev^{E_\alpha} h,g\}&=\left(\{f_{\mlambda}g\}-\mlambda^{E_\alpha}\{h_{\mlambda}g\}\right)\big|_{\mlambda=0}=\{f,g\}\\\notag
\{f,g+\mdev^{E_\alpha}h\}&=\left(\{f_{\mlambda}g\}+(\mlambda+\mdev)^{E_\alpha}\{f_{\mlambda}h\}\right)\big|_{\mlambda=0}\\
&=\{f,g\}+\mdev^{E_\alpha}\{f,h\}\sim\{f,g\}.
\end{align}
Part (b). The Jacobi property for the bracket follows immediately by setting $\mlambda=\mmu=0$ in PVA-Jacobi, while the skewsymmetry is a consequence of the skewsymmetry for the $\lambda$-bracket. First, we introduce a notation widely used in \cite{BdSK09}, namely
\begin{equation}\label{eq:expoform}
 \left(e^{\mdev\frac{\ud}{\ud\mlambda}}u\right)f(\mlambda)=f(\mlambda+\mdev)u.
\end{equation}
In words, we use the convention that the $\mdev$ in the exponent acts only on what is inside the parentheses. This notation is justified by the Taylor expansion of the exponential, which turns out to be equivalent to the RHS; the most important part is to always keep track of the terms on which the derivations are acting on.

We have
\begin{equation}
\begin{split}
\{g,f\}&=\{g_{\mlambda}f\}\big|_{\mlambda=0}\\
&=-_\to\{f_{-\mlambda-\mdev}g\}\big|_{\mlambda=0}\qquad\text{(skewsymmetry)}\\
&=-\left(e^{\mdev\frac{\ud}{\ud\mlambda}}\{f_{-\mlambda}g\}\right)\big|_{\mlambda=0}\qquad\text{using \eqref{eq:expoform}}\\
&=-\left(1+\mdev\frac{\ud}{\ud\mlambda}+\cdots\right)\{f_{-\mlambda}g\}\big|_{\mlambda=0}\\
&\sim-\{f,g\}.
\end{split}
\end{equation}
\end{proof}
Conversely, given a Poisson bracket among local densities, the corresponding $\lambda$-bracket is its formal Fourier transform. The aim of this paragraph is to show that the Fourier transform of the bracket of local densities is indeed a $\lambda$-bracket, which satisfies the PVA axioms if and only if the bracket is defined by a local Poisson bivector. This result is very important because working with the $\lambda$-brackets we do not deal with differential operators on a quotient space, but with simple differential polynomials.
\begin{deff}[Formal Fourier transform]
Given a $\hat{\A}$ valued formal distribution $D(\vec{x},\vec{y})$ (with $\vec{x},\vec{y}\in M$, $\dim M=d$), its \emph{formal Fourier transform} is the linear map
\begin{equation*}
 D(\vec{x},\vec{y})\mapsto\int\ud \vec{x} e^{\mlambda\cdot(\vec{x}-\vec{y})}D(\vec{x},\vec{y})=:\mathrm{F}D(\vec{y},\mlambda)
\end{equation*}
with values in $\hat{\A}[\lambda_1,\ldots,\lambda_d]$. It is equivalent to the one introduced, in a different context, by Kac and De Sole in \cite{dSK06}. The symbol of the integral $\int\ud\vec{x}$ must be regarded as the quotient operator with respect to $\sum_\alpha\dev_{x^\alpha}$.
\end{deff}
\begin{lemma}\label{thm:Symbol}
Let us consider a differential operator acting on a Dirac's delta
\begin{equation*}
 P(u(\vec{x}),\mdev_x)\delta(\vec{x}-\vec{y})=\sum_S P(u(\vec{x}))_S\mdev_x^S\delta(\vec{x}-\vec{y}).
\end{equation*}
Its formal Fourier transform is the symbol of the operator itself, namely
\begin{equation}\label{eq:Symbol}
\sum_S P(u(\vec{x}))_S\mlambda^S.
\end{equation}
\end{lemma}
\begin{proof}
Expanding the multiindex notation and keeping the sum implicit we have
\begin{equation*}
\begin{split}
\mathrm{F}P(u(\vec{y}),\mlambda)&=\int e^{\mlambda\cdot(\vec{x}-\vec{y})}P_{s_1\ldots s_d}(u(\vec{y}))\dev^{s_1}_{y^1}\ldots\dev^{s_d}_{y^d}\delta(\vec{x}-\vec{y})\ud\vec{x}\\
&=\int e^{\mlambda\cdot(\vec{x}-\vec{y})}P_{s_1\ldots s_d}(u(\vec{y}))(-\dev^{s_1}_{x^1})\ldots(-\dev^{s_d}_{x^d})\delta(\vec{x}-\vec{y})\ud\vec{x}\\
\intertext{integrating by parts}
&=\int\dev^{s_1}_{x^1}\ldots\dev^{s_d}_{x^d}e^{\mlambda\cdot(\vec{x}-\vec{y})}P_{s_1\ldots s_d}(u(\vec{y}))\delta(\vec{x}-\vec{y})\ud\vec{x}\\
&=\lambda_1^{s_1}\cdots\lambda_d^{s_d}P_{s_1\ldots s_d}(u(\vec{y}))\\
\intertext{which, using the usual multiindex notation, is}
&=P_S(u(\vec{y}))\mlambda^S.
\end{split}
\end{equation*}
\end{proof}
In order to prove our claim that the Fourier transform of a Poisson bracket of densities is a $\lambda$-bracket, we proceed as follows: first, we will prove that the skewsymmetry and the Jacobi property of the bracket among the generators, i.e. the coordinate functions, imply the skewsymmetyry \eqref{eq:LambdaSkew} and the PVA-Jacobi identity \eqref{eq:LambdaJac} for $\lambda$-bracket. Then we will compute the Fourier transform of the Poisson bracket between two generic densities and we will prove that it is expressed in terms of the Fourier transform of the bracket of generators by the master formula. Hence, the Fourier transform of the Poisson bracket is a $\lambda$-bracket.

The Poisson bracket of two coordinate functions $u^i(\vec{x})$ and $u^j(\vec{y})$ is given by
\begin{equation*}
 \{u^i(\vec{x}),u^j(\vec{y})\}=P^{ji}(u(\vec{y}))_S\mdev_y^S\delta(\vec{x}-\vec{y})
\end{equation*}
where $P^{ji}_S\mdev^S$ are the components of the Poisson bivector defining the bracket. From the lemma \ref{thm:Symbol}, its Fourier transform is
\begin{equation}\label{eq:LambdaGen}
 \{u^i_{\mlambda}u^j\}(\vec{y})=P^{ji}(u(\vec{y}))_S\mlambda^S.
\end{equation}
\begin{lemma}
The Lie bracket \eqref{eq:LambdaGen} is skewsymmetric in the sense of \eqref{eq:LambdaSkew}.
\end{lemma}
\begin{proof}
From the form of the Poisson brackets of generators we have that
\begin{align*}
\{u^i(\vec{x}),u^j(\vec{y})\}&=P^{ji}_S(u(\vec{y}))\mdev^S_y\delta(\vec{x}-\vec{y})\\
\{u^j(\vec{y}),u^i(\vec{x})\}&=P^{ij}_S(u(\vec{x}))\mdev^S_x\delta(\vec{y}-\vec{x}).
\end{align*}
We recall the skewsymmetry relation of the Poisson bivector \eqref{eq:antisBiv}, which gives
\begin{equation}\label{eq:antisBiv2}
 P^{ji}_S(\vec{y})=-\sum_T(-1)^{|T|}\binom{T}{S}\mdev^{T-S}P^{ij}_S(\vec{x})
\end{equation}
and apply it within the Fourier transform. We get
\begin{equation*}
\begin{split}
\{u^i_{\mlambda}u^j\}&=\int e^{\mlambda\cdot(\vec{x}-\vec{y})}P^{ji}_S(u(\vec{y}))\mdev^S_y\delta(\vec{x}-\vec{y})\ud \vec{x}\\
&=-\int e^{\mlambda\cdot(\vec{x}-\vec{y})}(-1)^{|T|}\binom{T}{S}\mdev^{T-S}(P^{ij}_S(u(\vec{x})))(-\mdev^S_x)\delta(\vec{x}-\vec{y})\ud \vec{x}\\
&=-\int(-1)^{|T|}\binom{T}{S}\mdev_x^S\left[e^{\mlambda\cdot(\vec{x}-\vec{y})}\mdev^{T-S}_xP^{ij}_T(u(\vec{x}))\right]\delta(\vec{x}-\vec{y})\ud\vec{x}\\
&=-\int(-1)^{|T|}\binom{T}{S}\binom{S}{L}\mlambda^L\mdev^{S-L+T-S}P^{ij}_T(u(\vec{x}))\delta(\vec{x}-\vec{y})\ud\vec{x}\\
&=-(-\mlambda-\mdev)^T P^{ij}_T(u(\vec{y}))\\
&=-_{\to}\{u^j_{-\mlambda-\mdev}u^i\}. 
\end{split}
\end{equation*}
\end{proof}
\begin{lemma}\label{lemma:PVAJacobi}
The Lie bracket \eqref{eq:LambdaGen} satisfies the PVA-Jacobi identity, namely
\begin{equation*}
\{u^i_{\mlambda}\{u^j_{\mmu}u^k\}\}-\{u^j_{\mmu}\{u^i_{\mlambda}u^k\}\}=\{\{u^i_{\mlambda}u^j\}{}_{\mmu+\mlambda}u^k\}.\end{equation*}
\end{lemma}
The proof of the lemma is a lenghty computation of the double formal Fourier transform with respect to $e^{\mlambda\cdot(\vec{x}-\vec{y})}e^{\mmu\cdot(\vec{y}-\vec{z})}$ for the three terms of the Jacobi identity, where the dependency of the coordinate functions $u^i(\vec{x})$, $u^j(\vec{y})$ and $u^k(\vec{z})$ on different independent variables plays a crucial role. The detailed account of the computations is left to Appendix \ref{app:proofPVAJacobi}.

To conclude this discussion, we want to show that taking the Fourier transform of the Poisson bracket between two densities gives a formula which coincides with the master formula for a $\lambda$-bracket. The computation is rather lengthy but in a sense straightforward. We want to compute
\begin{equation*}
 \int e^{\mlambda\cdot(\vec{x}-\vec{y})}\{f(\vec{x}),g(\vec{y})\}\ud \vec{x}.
\end{equation*}
For convenience, we drop the boldface typesetting to denote that the variables $x,y,z$ are coordinates in $\R$. We expand the Poisson bracket and get
\begin{equation*}
 \int e^{\mlambda\cdot(x-y)}\frac{\dev f(x)}{\dev u^i_M}\frac{\dev g(y)}{\dev u^j_N}\mdev_x^M\mdev_y^N\left(P^{ji}_S(y)\mdev_y^S\delta(x-y)\right)\ud x.
\end{equation*}
The derivatives respect to $x$ do not act on the coefficients $P^{ji}_S$ because they depend on functions of $y$ by definition. Inside the bracket, moreover, we can trade the derivatives respect to $y$ with the ones respect to $x$ exploiting the properties of Dirac's delta, thus obtaining
\begin{equation*}
 (-1)^{|M|}\int e^{\mlambda\cdot(x-y)}\frac{\dev f(x)}{\dev u^i_M}\frac{\dev g(y)}{\dev u^j_N}\mdev_y^N\left(P^{ji}_S(y)(-\mdev_x)^{M+S}\delta(x-y)\right)\ud x.
\end{equation*}
Then we perform the derivatives $\mdev_y^N$ and use the same trick
\begin{equation*}
 (-1)^{|M|}\binom{N}{T}\int e^{\mlambda\cdot(x-y)}\frac{\dev f(x)}{\dev u^i_M}\frac{\dev g(y)}{\dev u^j_N}\mdev_y^T(P^{ji}_S(y)(-\mdev_x)^{N-T+M+S}\delta(x-y)\ud x. 
\end{equation*}
We integrate by parts and let the $\mdev_x$ act properly. Then, we can finally integrate the Dirac's delta and get
\begin{equation*}
 (-1)^{|M|}\frac{\dev g}{\dev u^j_N}\binom{N}{T}\binom{N-T+M+S}{R}\mlambda^{N-T+M+S-R}\mdev^T(P^{ji}_S)\mdev^R\frac{\dev f}{\dev u^i_M}.
\end{equation*}
By applying the Newton's binomial,
\begin{multline*}
(-1)^{|M|}\frac{\dev g}{\dev u^j_N}\binom{N}{T}\mlambda^{N-T}\mdev^T(P^{ji}_S)(\mlambda+\mdev)^{M+S}\frac{\dev f}{\dev u^i_M} \\
=(-1)^{|M|}\frac{\dev g}{\dev u^j_N}(\mlambda+\mdev)^N(P^{ji}_S(\mlambda+\mdev)^{M+S}\frac{\dev f}{\dev u^i_M})\\
=\frac{\dev g}{\dev u^j_N}(\mlambda+\mdev)^N(P^{ji}_S(\mlambda+\mdev)^S(-\mlambda-\mdev)^M\frac{\dev f}{\dev u^i_M}).
\end{multline*}
Recalling the form of the $\lambda$-bracket between the generators, the last expression is
\begin{equation*}
 \frac{\dev g}{\dev u^j_N}(\mlambda+\mdev)^N\{u^i_{\mlambda+\mdev}u^j\}(-\mlambda-\mdev)^M\frac{\dev f}{\dev u^i_M}),
\end{equation*}
namely the master formula.

We have thus proved the following theorem
\begin{thm}\label{thm:FourPVA}
 Given a local Poisson bivector $P$ on the space of maps $\mathrm{Map}(\Sigma,M)\cong\hat{\A}$, the Fourier transform of the bracket induced by the bivector is the $\lambda$-bracket of a Poisson Vertex Algebra on $\hat{\A}$.
\begin{equation}\label{eq:lambdaTransf}
 \{f_{\mlambda}g\}(\vec{y}):=\int_\Sigma e^{\mlambda\cdot(\vec{x}-\vec{y})}\{f(\vec{x}),g(\vec{y})\}\ud\vec{x}.
\end{equation}
\end{thm}
\subsection{Cohomology of Poisson Vertex Algebras}\label{ssec:PVACohom}
In Section \ref{ssec:mPVA} we have proved the correspondence between local Poisson bivectors and PVAs. It is well known \cite{L77} that from the Schouten relation $[P,P]=0$ it follows that one can define a linear differential $\ud_P=[P,\cdot]$ which is a coboundary operator, $\ud_P^2=0$.

From the properties of the Schouten bracket it follows that $\ud_P\colon\Lambda^k(M)\to\Lambda^{k+1}(M)$, both for the finite and the infinite dimensional setting (for the latter, see for instance \cite{DZ}). One can define the cochain complex
\begin{equation*}
 0\to \Lambda^0(M)\xrightarrow{\ud_P}\Lambda^1(M)\xrightarrow{\ud_P}\Lambda^2(M)\xrightarrow{\ud_P}\cdots
\end{equation*}
and its cohomology, which is called the Poisson--Lichnerowicz cohomology of $M$.

It is quite natural to repeat the construction in the context of Poisson Vertex Algebras. All the details are exposed by De Sole and Kac for one-dimensional PVAs \cite{dSK13-2}, but the definitions we are interested in are basically the same.
\begin{deff}[Variational complex of a PVA]
Given an algebra of differential polynomials $\hat{\A}$, let us consider the free commutative superalgebra $\widetilde{\Omega}^\bullet(\hat{\A})$ over $\hat{\A}$ with odd generators $\delta u_M^i$, $i=1,\ldots,n$, $M\in\Z^d_{\geq0}$. We define a grading on $\widetilde{\Omega}^\bullet(\hat{\A})$, imposing $\deg f=0$ for $f\in\hat{\A}$ and $\deg\delta u^i_M=1$, so that we can decompose $\widetilde\Omega^\bullet=\bigoplus_{k\in Z_{\geq 0}}\widetilde\Omega^k$. An odd derivation $\delta\colon\widetilde{\Omega}^k\to\widetilde{\Omega}^{k+1}$ is defined by
\begin{align*}
 \delta f&:=\sum_{i,M}\frac{\dev f}{\dev u^i_M}\delta u^i_{M}&f&\in\hat{\A}\\
\delta(\delta u^i_M)&=0
\end{align*}
Since $\delta^2=0$, we get a complex $(\widetilde{\Omega},\delta)$ which is called the \emph{basic variational complex}. The total derivatives $\dev_\alpha$ defined in \eqref{eq:totDer} can be extended to $\widetilde{\Omega}$ by $\dev_\alpha\delta u^i_M=\delta u^i_{M+E_\alpha}$. One can easily prove that $[\dev_\alpha, \delta]=0$ for any $\alpha=1,\ldots,d$. In such a way, we can define the reduced complex $(\Omega,\delta)$, that is called the \emph{variational complex}.
\begin{equation}\label{eq:varComplex}
 \Omega=\oplus_{k\geq 0}\Omega^k,\qquad\Omega^k=\widetilde{\Omega}^k/\sum_\alpha\dev_\alpha\widetilde{\Omega}^k
\end{equation}
and $\delta$ is the induced differential between the quotient spaces.
\end{deff}
The space $\Omega^0(\hat{\A})$ is immediately identified with $\hat{\F}$ the space of local functionals. The derivative $\delta$ is the operator of the variational derivative, since we can use $\delta u^i_M=\mdev^M\delta u^i$ and the quotient map to get exactly \eqref{eq:VarDer}.
$\Omega^1(\hat{\A})$ is the space of evolutionary vector fields and the space $\Omega^2(\hat{\A})$ is isomorphic to the space of local bivectors (see the proof in \cite{BdSK09}, which holds for all $\Omega^k$ and local $k$-vectors).

Let us consider an element $X\in\Omega^{k-1}(\hat{\A})$. According to the aforementioned result, we regard $\Omega^k$ as the space of local $k$-vector fields. Given a local bivector whose symbol is $\sum P_S\mlambda^S$, we extend to the multidimensional PVA the definition of \cite{dSK13-2} for the PVA differential and we get
\begin{multline*}
(\ud_P X)\big(F^0,\dots,F^k\big)=\sum_{i=0}^k (-1)^{k+i}
\int F^i\cdot P_S\mdev^S\frac{\delta}{\delta u} X(F^0,\stackrel{i}{\check{\dots}},F^k)+\\
+ \sum_{0\leq i<j\leq k} (-1)^{k+i+j} X\Big(
\frac\delta{\delta u} \int \big(F^i\cdot P_S\mdev^S F^j\big),
F^0,\stackrel{i}{\check{\dots}}\,\stackrel{j}{\check{\dots}},F^k
\Big)
\end{multline*}
The differential $\ud_P\colon\Omega^{k-1}\to\Omega^k$ is in one-to-one correspondence with $[P,\cdot]$ the ordinary coboundary operator of the Poisson--Lichnerowicz cohomology, and it squares to 0 thanks to the PVA-Jacobi identity.

Exactly as in the classical case, one can then define the cohomology of the complex: the first cohomology group $H^1(\hat{\A},\ud_P)$ is identified with the symmetries of the Poisson bivector and the second cohomology group $H^2(\hat{\A},\ud_P)$ is identified with the Poisson structures compatible with the structure $P$. In the language of the $\lambda$-brackets, we say that $H^2(\hat{\A},\ud_{\{\cdot_{\mlambda}\cdot\}_0})$ is the space of the $\lambda$-brackets which are compatibles with $\{\cdot_{\mlambda}\cdot\}_0$.

We introduce a grading on $\hat{\A}$. We simply define
\begin{align*}
 \deg u^i_I&=|I|:=\sum_\alpha I_\alpha&\deg f(u_i)&=0.
\end{align*}
This means that $\deg\dev_\alpha f(u^i;u^i_I)=\deg f(u^i;u^i_I)+1$. Moreover, we assign to $\mlambda^I$ the degree $|I|$. Any $\lambda$-bracket can be decomposed, according to this grading, into homogenous parts. A rigorous way to obtain such a decomposition is to introduce a formal indeterminate $\epsilon$ of degree $-1$ and rescale the bracket in such a way that $\sum_{k=-1}^\infty\epsilon^{k+1}\{u^i_{\mlambda}u^j\}^{[k]}$ is of degree 0; the reason why $\{u^i_{\mlambda}u^j\}^{[k]}$ is in fact of degree $k+1$ is that we want to keep consistency with the notation used in \cite{DZ}.

We can use this grading to decompose each cohomology group
\begin{equation*}
 H^k(\hat{\A},\ud_{\{\cdot_{\mlambda}\cdot\}_0})=\bigoplus_n H^k_{[n]}(\hat{\A},\ud_{\{\cdot_{\mlambda}\cdot\}_0}).
\end{equation*}
We will call each $H^k_{[n]}$ the $k$-th cohomology group of $n$-th order.

\section{Multidimensional Poisson brackets of hydrodynamic type}\label{sec:mHYPB}
In this section we apply the formalism we have discussed in the previous one to the so-called multidimensional Poisson brackets of hydrodynamic type introduced by Dubrovin and Novikov in \cite{DN84}. They are brackets on a space of maps $\Sigma\to M$ defined by a bivector whose components are differential operators of the first order linear with respect to the first derivatives of the maps; in terms of $\lambda$-bracket among the generators of $\hat{\A}$, they have the form
\begin{equation}\label{eq:DNbiv}
 \{u^i_{\mlambda}u^j\}=\sum_{\alpha=1}^d g^{ij\alpha}(u)\lambda_\alpha+\sum_{\stackrel{\alpha=1\ldots d}{k=1\ldots n}}b^{ij\alpha}_k(u)\dev_\alpha u^k.
\end{equation}
Since the order of derivatives we will deal with is not very high, it is easier to switch back to a single-index notation, namely $\mlambda^{E_\alpha}=\lambda_\alpha$ and (for instance) $u_{E_\alpha+E_\beta}=\dev_{\alpha\beta}u$. Dubrovin and Novikov in \cite{DN84} have found, generalizing the result for the one-dimensional case they discovered in \cite{DN83}, a set of necessary conditions for a differential operator of type \eqref{eq:DNbiv} to define a Poisson bracket in the space of local functionals, provided $g^\alpha$ were nondegenerate. Some years later Mokhov \cite{M88} proved the complete set of axioms the collection of functions $(g^{ij\alpha},b^{ij\alpha}_k)$ must fulfil. They are summarized in the following theorem
\begin{thm}[\cite{M88}]\label{thm:Mokhov}
 Let $P$ be a differential operator whose symbol is \eqref{eq:DNbiv}. The bracket among local functionals of density $f,g$ defined by
 \begin{equation*}
  \left\{\int f,\int g\right\}:=\int\frac{\delta f}{\delta u^i}P^{ij}\frac{\delta g}{\delta u^j}
 \end{equation*}
is a Poisson bracket -- equivalently, \eqref{eq:DNbiv} is the $\lambda$-bracket of a Poisson Vertex Algebra -- if and only if
\begin{subequations}\label{eq:Mokhov}
\begin{gather}\label{eq:M1}
g^{ij\alpha}=g^{ji\alpha}\\\label{eq:M2}
\frac{\dev g^{ij\alpha}}{\dev u^k}=b^{ij\alpha}_k+b^{ji\alpha}_k\\\label{eq:M3}
\sum_{(\alpha,\beta)}\left(g^{ai\alpha}b^{jk\beta}_a-g^{aj\beta}b^{ik\alpha}_a\right)=0\\ \label{eq:M4}
\sum_{(i,j,k)}\left(g^{ai\alpha}b^{jk\beta}_a-g^{aj\beta}b^{ik\alpha}_a\right)=0\\ \label{eq:M5}
\sum_{(\alpha,\beta)}\left[g^{ai\alpha}\left(\frac{\dev b^{jk\beta}_a}{\dev u^r}-\frac{\dev b^{jk\beta}_r}{\dev u^a}\right)+b^{ij\alpha}_ab^{ak\beta}_r-b^{ik\alpha}_ab^{aj\beta}_r\right]=0\\ \label{eq:M6}
g^{ai\beta}\frac{\dev b^{jk\alpha}_r}{\dev u^a}-b^{ij\beta}_ab^{ak\alpha}_r-b^{ik\beta}_ab^{ja\alpha}_r=g^{aj\alpha}\frac{\dev b^{ik\beta}_r}{\dev u^a}-b^{ja\alpha}_ab^{ak\beta}_r-b^{jk\alpha}_ab^{ia\beta}_r\\ \notag
\frac{\dev}{\dev u^s}\left[g^{ai\alpha}\left(\frac{\dev b^{jk\beta}_a}{\dev u^r}-\frac{\dev b^{jk\beta}_r}{\dev u^a}\right)+b^{ij\alpha}_ab^{ak\beta}_r-b^{ik\alpha}_ab^{aj\beta}_r\right]\\ \label{eq:M7}
+\frac{\dev}{\dev u^r}\left[g^{ai\beta}\left(\frac{\dev b^{jk\alpha}_a}{\dev u^s}-\frac{\dev b^{jk\alpha}_s}{\dev u^a}\right)+b^{ij\beta}_ab^{ak\alpha}_s-b^{ik\beta}_ab^{aj\alpha}_s\right]\\\notag
+\sum_{(i,j,k)}\left[b^{ai\beta}_r\left(\frac{\dev b^{jk\alpha}_s}{\dev u^a}-\frac{\dev b^{jk\alpha}_a}{\dev u^s}\right)\right]+\sum_{(i,j,k)}\left[b^{ai\alpha}_s\left(\frac{\dev b^{jk\beta}_r}{\dev u^a}-\frac{\dev b^{jk\beta}_a}{\dev u^r}\right)\right]=0
\end{gather}\end{subequations}
The notation $\sum_{(a_1,a_2,\ldots)}$ used for instance in \eqref{eq:M3} means the cyclic summation over the indices. Conditions \eqref{eq:M1} -- \eqref{eq:M2} are equivalent to the skewsymmetry of the bracket, while the other ones are equivalent to the validity of the Jacobi identity.
\end{thm}
\begin{proof}
We explicitly impose the skewsymmetry condition \eqref{eq:LambdaSkew} and the PVA-Jacobi identity \eqref{eq:LambdaJac} for the bracket \eqref{eq:DNbiv} among three generators of $\hat{\A}$. The vanishing of the first degree terms in $\lambda_\alpha$ for \eqref{eq:LambdaSkew} are the conditions \eqref{eq:M1}, while the vanishing of the coefficients of $u^k_\alpha$ are \eqref{eq:M2}. We then use the master formula to compute \eqref{eq:LambdaJac}. It gives a degree 2 differential polynomial in the $\lambda$'s and the $\mu$'s. The remaining conditions are the vanishing of the coefficients for, respectively, $\lambda_\alpha\lambda_\beta$, $\lambda_\alpha\mu_\beta$ (the coefficients for $\mu\leftrightarrow\lambda$ are equivalent, provided the skewsymmetry), $u^r_{\alpha\beta}$, $u^r_\alpha\lambda_\beta$, and $u^r_\alpha u^s_\beta$.
\end{proof}
\subsection{Symmetries of the Poisson brackets of hydrodynamic type}\label{ssec:symHYPB}
\begin{deff}A \emph{$n$-th order symmetry} of a PVA $(\hat{\A},\{\cdot_{\mlambda}\cdot\})$ is an evolutionary vector field $\xi\in\mathrm{Der}(\hat{\A})$, $[\xi,\dev_\alpha]=0$ $\forall \alpha$, with the following properties:
\begin{enumerate}
\item The components $X^i$ of the vector field (see \eqref{eq:evVF_PVA}) are homogeneous differential polynomials of order $n$; 
\item $\xi\left(\{f_{\mlambda}g\}\right)=\{\xi(f)_{\mlambda}g\}+\{f_{\mlambda}\xi(g)\}$
\end{enumerate}
\end{deff}
\begin{deff}\label{def:HamVF}A \emph{Hamiltonian vector field} in the context of PVA \cite{BdSK09} is an evolutionary vector field $\xi_H$ whose components $X^i$ are
\begin{equation}
 X^i(u_J)=\xi_H(u^i)=\{H_{\mlambda}u^i\}\vert_{\mlambda=0}
\end{equation}
for $H\in\hat{\A}$.
\end{deff}
In terms of PVA cohomology, a symmetry is a cocycle in $\Omega^1(\hat{\A},\{\cdot_{\mlambda}\cdot\})$ and a Hamiltonian vector field is a coboundary in the same space.

The fact that a Hamiltonian vector field, in the terms we defined it, is a symmetry of the $\lambda$-bracket is easily obtained from the PVA-Jacobi \eqref{eq:LambdaJac} identity after setting $\mlambda=0$
\begin{equation}
\begin{split}
\{H_{\mlambda}\{f_{\mmu}g\}\}\vert_{\mlambda=0}&=\left(\{f_{\mmu}\{H_{\mlambda}g\}\}+\{\{H_{\mlambda}f\}_{\mlambda+\mmu}g\}\right)\vert_{\mlambda=0}\\
\xi_H(\{f_{\mmu}g\})&=\{f_{\mmu}\xi_H(g)\}+\{\xi_H(f)_{\mmu}g\}
\end{split}
\end{equation}
On the other hand, the classification of the symmetries of the $\lambda$-bracket allows us to characterize the first PVA-cohomology group and, in the terms of section \ref{ssec:PVACohom}, the Poisson-Lichnerowicz cohomology of the associated Poisson bracket.

In the rest of this section, we investigate the first order symmetries of the Poisson brackets of hydrodynamic type for $d=n=2$. We will denote the generators of the algebra of differential polynomials $\hat{\A}$ as $(p_1\equiv p, p_2\equiv q)$. Ferapontov and collaborators provide in \cite{FOS11} a classification, based on Mokhov's results \cite{M08}, of all the undeformed Poisson structures on such a space up to Miura transformations and linear change of the independent variables. They are, in terms of $\lambda$-brackets,
\begin{align}\label{eq:defP1}
 \{{p_i}_{\mlambda}{p_j}\}_1&=\delta_{ij}\lambda_i\\ \label{eq:defP2}
 \{{p_i}_{\mlambda}{p_j}\}_2&=\delta_{i+j,3}\lambda_1+\delta_{ij}\delta_{j2}\lambda_2\\ \label{eq:LPPVA}
\left\{{p_i}_{\mlambda}{p_j}\right\}_{LP}&=-\left(p_i\lambda_j+p_j\lambda_i+\dev_ip_j\right).
\end{align}
The essential difference, under the point of view of the classification, is that the two strucures $\{\cdot_{\mlambda}\cdot\}_1$ and $\{\cdot_{\mlambda}\cdot\}_2$ are constant, while $\{\cdot_{\mlambda}\cdot\}_{LP}$ is not. Mokhov observed that the brackets that are essentially nonconstant can always be written in a coordinate system for which they are of form $\{\cdot_{\mlambda}\cdot\}_{LP}$,a  structure that has been introduced by Novikov in \cite{N82} as a Lie--Poisson bracket of hydrodynamic type.

\paragraph{Lie--Poisson brackets of hydrodynamic type}
We consider the Lie algebra $\mathfrak{g}=\X(\Sigma)$ of the vector fields on a manifold, let us say a $d$-dimensional torus. It has been known for long time that this algebra is tightly related to the Euler's equation for ideal fluids (\cite{Ar66}). In some coordinates such vector fields can be written as $X(\vec{x})=\sum X^i(\vec{x})\dev_i$, $i=1\ldots d$; the components of their commutator are $[X,Y]^i(\vec{x})=\sum X^j(\vec{x})\dev_j Y^i(\vec{x})-Y^j(\vec{x})\dev_jX^i(\vec{x})$. This implies that the structure functions of $\mathfrak{g}$ must have the form $C^i_{jk}(\vec{x},\vec{y},\vec{z})=\delta^i_j\delta(\vec{z}-\vec{x})\dev_k\delta(\vec{y}-\vec{z})-\delta^i_k\delta(\vec{y}-\vec{x})\dev_j\delta(\vec{z}-\vec{y})$. It is well known that, given a Lie algebra $\mathfrak{g}$, it is always possible to endow its dual space $\mathfrak{g}^*$ with a Poisson bracket called the \emph{Lie--Poisson bracket}. In this setting, the coordinates on $\mathfrak{g}^*$ are a set of functions $p_i(\vec{x})$ such that
\begin{equation*}
 \int p_i(\vec{x})v^i(\vec{x})\ud\vec{x}
\end{equation*}
behaves as a scalar under change of variables. Here, $v^i(\vec{x})$ are the components of a vector field. This means that $p_i(\vec{x})$ are densities of 1-forms. The Lie--Poisson bracket is linear in the coordinates and defined by the structure functions as
\begin{equation}\label{eq:LPbra}
 \{p_j(\vec{y}),p_k(\vec{z})\}=\int C^i_{jk}(\vec{x},\vec{y},\vec{z})p_i(\vec{x})\ud\vec{x}.
\end{equation}
In the same fashion as the general case we dealt with in the previous section, we define the $\lambda$-bracket for the generators of the differential polynomial algebra $C^{\infty}(p_i)[p_{iI}]$. Regarding the $p$'s and their derivatives in \eqref{eq:LPbra} as independent variables, in the spirit of jet bundles and as we defined in Section \ref{ssec:mapspace}, we drop their dependence on the points of $\Sigma$. For $d=2$, we get \eqref{eq:LPPVA} as the result. Notice that the form of the Lie--Poisson $\lambda$-bracket would be the same for any $d=n$.
\paragraph*{}
Let us consider each of the normal forms for the $\lambda$-bracket of hydrodynamic type. We will compute the action of an evolutionary vector field on the brackets between two generators in order to characterize the conditions it must satisfy in order to be a symmetry of the brackets themselves. We restrict ourselves to consider only first order vector fields, whose components can be written as
\begin{equation}\label{eq:VField1Ord}
 \xi(p_i)=X_i(p,p_I)=A^{ab}(p)\dev_ap_b
\end{equation}
where each index runs from 1 to $d=n=2$ and we follow the Einstein convention for the sum over repeated indices.

The conditions for $\xi$ to be a symmetry can be directly computed and are summarized in the following lemmas.
\begin{lemma}\label{thm:SymP1}
An evolutionary vector field of form \eqref{eq:VField1Ord} is a first order symmetry of the bracket \eqref{eq:defP1} if and only if the following conditions hold:
\begin{subequations}\label{eq:1stOrderSymP1}
\begin{align}\label{eq:1stOrderSymP1a}
A^{ab}_j\delta^b_i+A^{ba}_j\delta^a_i-A^{ab}_i\delta^b_j-A^{ba}_i\delta^a_j&=0\\\label{eq:1stOrderSymP1b}
\frac{\dev A^{bl}_i}{\dev p_j}\delta^a_j-\frac{\dev A^{aj}_i}{\dev p_l}\delta^b_j-\frac{\dev A^{bj}_i}{\dev p_l}\delta^a_j+\frac{\dev A^{bl}_j}{\dev p_i}\delta^a_i&=0\\\label{eq:1stOrderSymP1c}
\frac{\dev^2A^{al}_i}{\dev p_j\dev p_m}\delta^b_j+\frac{\dev^2A^{bm}_i}{\dev p_j\dev p_l}\delta^a_j-\frac{\dev^2A^{aj}_i}{\dev p_l\dev p_m}\delta^b_j-\frac{\dev^2A^{bj}_i}{\dev p_l\dev p_m}\delta^a_j&=0\\\label{eq:1stOrderSymP1d}
\frac{\dev A^{al}_i}{\dev p_j}\delta^b_j+\frac{\dev A^{bl}_i}{\dev p_j}\delta^a_j-\frac{\dev A^{aj}_i}{\dev p_l}\delta^b_j-\frac{\dev A^{bj}_i}{\dev p_l}\delta^a_j&=0
\end{align}
\end{subequations}
In particular, for the case $d=n=2$, the solutions of \eqref{eq:1stOrderSymP1} for $A^{ab}_i(p,q)$ are
\begin{equation}\label{eq:1stOrderSymP1EXPL}
 \begin{split}
   A^{11}_1=\frac{\dev^2 K}{\dev p^2}\qquad\qquad&A^{11}_2=0\\
A^{12}_1=\frac{\dev^2 K}{\dev p\dev q}\qquad\qquad&A^{12}_2=c_1\\
A^{21}_1=c_2\qquad\qquad&A^{21}_2=\frac{\dev^2 K}{\dev p\dev q}\\
A^{22}_1=0\qquad\qquad&A^{22}_2=\frac{\dev^2 K}{\dev q^2}
 \end{split}
\end{equation}
where $c_1$ and $c_2$ are constants and $K=K(p,q)$ is a generic function of the coordinates.
\end{lemma}
\begin{proof}Recalling that $\xi(p_i)=A^{ab}_i\dev_ap_b$ and $\xi(f)=\mdev^I\left(A^{ab}_i\dev_ap_b\right)\dev f/\dev p_{iI}$ we compute $\xi(\{{p_i}_{\mlambda}p_j\}-\{A^{ab}_i\dev_ap_b{}_{\mlambda}p_j\}-\{p_i{}_{\mlambda}A^{ab}_j\dev_ap_b\}$ and set to 0 the coefficients of the four terms $\lambda_a\lambda_b$, $\lambda_a\dev_bp_l$, $\dev_ap_l\dev_bp_m$, and $\dev_{ab}p_l$. This procedure gives the set of equations \eqref{eq:1stOrderSymP1}; for the case $d=n=2$ we explicitly write down the algebraic equations \eqref{eq:1stOrderSymP1a} that imply $A^{11}_2=0$, $A^{22}_1=0$, and $A^{12}_1=A^{21}_2$. The complete solution is then easily found using \eqref{eq:1stOrderSymP1b}. Indeed, equations \eqref{eq:1stOrderSymP1d} turn out to be equivalent to \eqref{eq:1stOrderSymP1b} and \eqref{eq:1stOrderSymP1c} are differential consequences of that.
\end{proof}
\begin{lemma}\label{thm:SymP2}
 An evolutionary vector field of form \eqref{eq:VField1Ord} is a first order symmetry of the bracket \eqref{eq:defP2} if and only if the following conditions hold:
\begin{equation}\label{eq:1stOrderSymP2}
\begin{split}
A^{12}_2=A^{11}_1\qquad\qquad&A^{22}_1=0\\
A^{12}_1+A^{21}_1=A^{22}_2\qquad\qquad&\\
A^{21}_2-A^{11}_1=c_1\qquad\qquad&A^{12}_1-A^{22}_2=c_2\\
\frac{\dev A^{11}_2}{\dev q}-\frac{\dev A^{21}_2}{\dev p}=0\qquad\qquad&\frac{\dev A^{21}_2}{\dev q}-\frac{\dev A^{22}_2}{\dev p}=0
\end{split}
\end{equation}
\end{lemma}
\begin{proof}
The particular form of the bracket \eqref{eq:defP2} makes explicitly computing the symmetry condition in the case $d=n=2$ the most effective approach to the problem. As in Lemma \ref{thm:SymP1}, the coefficients of the terms $\lambda_a\lambda_b$ are algebraic equations, while the coefficients of the other terms are linear PDEs. Hence we can first reduce the number of unknowns for the differential equations. In order to simplify the set of the remaining equations, we relied to a powerful computational tool which is called a \emph{Janet basis} for the linear system of PDEs \cite{PR05}. It provides the normal form for the system, which is unique up to the ordering of variables, and can be computed using the Maple package \texttt{Janet} \cite{BC03}.
\end{proof}
In this particular case, the original system is not very complicated and could be simplified also by hand. Nevertheless, the search for a Janet basis of a linear system of PDEs will be necessary for the more involved systems we will consider next.

With the same technique used to prove Lemma \ref{thm:SymP1}, and then computing the Janet basis for the explicit formula of the symmetry conditions as we did in Lemma \ref{thm:SymP2}, we can characterize the symmetries of \eqref{eq:LPPVA}. The results are summarized in the following lemma.
\begin{lemma}\label{thm:SymLP}
An evolutionary vector field of form \eqref{eq:VField1Ord} is a first order symmetry of the bracket \eqref{eq:LPPVA} if and only if the following conditions hold:
\begin{subequations}\label{eq:1stOrderSymLP}
\begin{align}
A^{al}_j\frac{\dev}{\dev p_b}\left(p_ip_l\right)-A^{al}_i\frac{\dev}{\dev p_b}\left(p_jp_l\right)+A^{bl}_j\frac{\dev}{\dev p_a}\left(p_ip_l\right)-A^{bl}_i\frac{\dev}{\dev p_a}\left(p_jp_l\right)&=0\\\notag
A^{ba}_j\delta^l_i+A^{al}_j\delta^b_i-A^{bl}_i\delta^a_j-A^{ab}_i\delta^l_j+\frac{\dev A^{bl}_i}{\dev p_m}\frac{\dev}{\dev p_a}\left(p_mp_j\right)+\qquad\qquad&\\
-\frac{\dev A^{am}_i}{\dev p_l}\frac{\dev}{\dev p_b}\left(p_mp_j\right)-\frac{\dev A^{bm}_i}{\dev p_l}\frac{\dev}{\dev p_a}\left(p_mp_j\right)-\frac{\dev A^{bl}_j}{\dev p_m}\frac{\dev}{\dev p_a}\left(p_ip_m\right)&=0\\
\left(\frac{\dev^2A^{al}_i}{\dev p_s\dev p_m}-\frac{\dev^2A^{as}_i}{\dev p_l\dev p_m}\right)p_s\delta^b_j+\left(\frac{\dev^2A^{bm}_i}{\dev p_s\dev p_l}-\frac{\dev^2A^{bs}_i}{\dev p_l\dev p_m}\right)p_s\delta^a_j+&\\
\left(\frac{\dev^2A^{bm}_i}{\dev p_a\dev p_l}-\frac{\dev^2A^{ab}_i}{\dev p_l\dev p_m}\right)p_j+\left(\frac{\dev^2A^{al}_i}{\dev p_b\dev p_m}-\frac{\dev^2A^{ba}_i}{\dev p_l\dev p_m}\right)p_j+&\\\notag
+\left(\frac{\dev A^{al}_i}{\dev p_b}-\frac{\dev A^{ab}_i}{\dev p_l}\right)\delta^m_j+\left(\frac{\dev A^{bm}_i}{\dev p_a}-\frac{\dev A^{ba}_i}{\dev p_m}\right)\delta^l_j&=0\\\notag
\left(\frac{\dev A^{al}_i}{p_s}-\frac{\dev A^{as}_i}{\dev p_l}\right)\frac{\dev}{\dev p_b}\left(p_sp_j\right)+\left(\frac{\dev A^{bl}_i}{p_s}-\frac{\dev A^{bs}_i}{\dev p_l}\right)\frac{\dev}{\dev p_a}\left(p_sp_j\right)&=0
\end{align}
\end{subequations}
For the case $d=n=2$, the algebraic equations allow us to express $A^{21}_1$, $A^{12}_2$, and $A^{21}_2$ in terms of the remaining five functions; we then compute the Janet basis of the remaining set of linear PDEs. After this procedure, \eqref{eq:1stOrderSymLP} in the two-dimensional case is reduced to the system
\begin{subequations}
\begin{align*}
 A^{21}_1&=A^{22}_2-\frac{2q}{p}A^{22}_1&A^{12}_2&=A^{11}_1-\frac{2p}{q}A^{11}_2\\
A^{21}_2&=\frac{p^3A^{11}_2+pq^2A^{12}_1-q^3A^{22}_1}{p^2q}
\end{align*}
\begin{align}\label{eq:1stOrderSymLPdiff}
 2\frac{\dev A^{22}_2}{\dev q}p+ \frac{\dev A^{12}_1}{\dev q} p+3\frac{\dev A^{22}_1}{\dev q}q+5 A^{22}_1&=0\\
\frac{\dev A^{22}_1}{\dev p}p+2\frac{\dev A^{22}_1}{\dev q}q+2A^{22}_1 - \frac{\dev A^{22}_2}{\dev q}p&=0\\
\frac{\dev A^{11}_1}{\dev p}q -2\frac{\dev A^{11}_2}{\dev p}p -\frac{\dev A^{11}_2}{\dev q}q -2A^{11}_2&=0
\end{align}
\begin{align}
\frac{\dev A^{11}_1 }{\dev q}p^2q^2-2\frac{\dev A^{11}_2}{\dev q}p^3q+2\frac {\dev A^{22}_1}{\dev q}q^4-\frac{\dev A^{22}_2}{\dev q}pq^3+&\\\notag
2A^{11}_2p^3+4A^{22}_1q^3-A^{12}_1pq^2&=0\\
4\frac{\dev A^{22}_1}{\dev q}q^4-\frac{\dev A^{11}_2}{\dev q}p^3q+\frac{\dev A^{22}_2}{\dev p}p^2q^2-2\frac{\dev A^{22}_2}{\dev q}pq^3+&\\\notag
A^{11}_2p^3-A^{12}_1pq^2+7A^{22}_1 q^3&=0\\
\frac{\dev A^{12}_1}{\dev p}p^2q^2-2\frac{\dev A^{11}_2}{\dev q}p^3q+2\frac{\dev A^{22}_1}{\dev q}p^4-\frac{\dev A^{22}_2}{\dev q}pq^3+&\\\notag
2A^{11}_2p^3-A^{12}_1pq^2+4A^{22}_1q^3&=0
\end{align}
\end{subequations}
\end{lemma}
We shift now into considering the form of the Hamiltonian vector fields for each normal form of the $\lambda$-bracket. From Definition \ref{def:HamVF} we can explicitly compute the coefficients $A^{ab}_i$ of a first order Hamiltonian vector field, provided that the Hamiltonian $h$ is a function of $(p, q)$ only. Hamiltonian functions homogeneous of differential degree $p$ produce, for brackets of hydrodynamic type, vector fields of order $p+1$.

The Hamiltonian vector fields of the bracket \eqref{eq:defP1} can be immediately compared with \eqref{eq:1stOrderSymP1EXPL}, since for this easy case we have explicitly solved the symmetry conditions. For a Hamiltonian $h(p,q)$, the associated Hamiltonian vector field has form
\begin{equation}\label{eq:HamVF_P1}
 \xi_h(p_i)=A^{ab}_i\dev_ap_b=\frac{\dev^2 h}{\dev p_i\dev p_b}\delta^a_i\dev_ap_b.
\end{equation}
In particular, this means that the first cohomology group for $\{\cdot_{\mlambda}\cdot\}_1$ is not trivial, as opposite as what is known for 1-dimensional Poisson brackets of hydrodynamic type \cite{G01}. Indeed, there exists a family of non Hamiltonian symmetry depending on two arbitrary constants $c_1$ and $c_2$.

The Hamiltonian vector fields of the bracket \eqref{eq:defP2} can be easily computed, too. Their components are given in terms of the coefficients $A^{ab}_i$, which for an Hamiltonian $h(p,q)$ are
\begin{equation}\label{eq:HamVF_P2}
\begin{split}
 A^{11}_1=A^{12}_2=A^{21}_2&=\frac{\dev^2 h}{\dev p\dev q}\\
A^{21}_1=A^{22}_1&=0\\
A^{11}_2&=\frac{\dev^2 h}{\dev p^2}\\
A^{12}_1=A^{22}_2&=\frac{\dev^2h}{\dev q^2}
\end{split}
\end{equation}
There exist solutions of \eqref{eq:1stOrderSymP2} with $c_1,c_2\neq 0$, but such solutions are not Hamiltonian vector fields. For example, the vector field whose components are $A^{11}_1=A^{11}_2=A^{12}_2=A^{22}_1=A^{22}_2=0$, $A^{21}_2=c_1$, and $A^{12}_1=-A^{21}_1=c_2$ is a symmetry but cannot be a Hamiltonian vector field. Thus, the first cohomology group of $\{\cdot_{\mlambda}\cdot\}_2$ is not trivial.

In order to characterize the first cohomology group for $\{\cdot_{\mlambda}\cdot\}_{LP}$ we choose to proceed in a different way. The algebraic equations in the set of conditions \eqref{eq:1stOrderSymLP} allowed us to express the linear PDEs \eqref{eq:1stOrderSymLPdiff} in terms of 5 out of the 8 coefficients $A^{ab}_i$. We compute these 5 coefficients for a Hamiltonian vector field. They are
\begin{subequations}\label{eq:HamVF_LP}
 \begin{align}
A^{11}_1&=-\left(\frac{\dev h}{\dev p}+2p\frac{\dev^2 h}{\dev p^2}+q\frac{\dev^2 h}{\dev p\dev q}\right)\\
A^{11}_2&=-q\frac{\dev^2 h}{\dev p^2}\\
A^{12}_1&=-\left(2p\frac{\dev^2 h}{\dev p\dev q}+q\frac{\dev^2 h}{\dev q^2}\right)\\
A^{22}_1&=-p\frac{\dev^2 h}{\dev q^2}\\
A^{22}_2&=-\left(\frac{\dev h}{\dev q}+2q\frac{\dev^2 h}{\dev q^2}+\frac{\dev^2 h}{\dev p\dev q}\right) 
 \end{align}
\end{subequations}
We can regard the five equations \eqref{eq:HamVF_LP} as an overdetermined system of inhomogeneous linear PDEs for the unknown function $h$. The compatibility conditions for the functions $A^{ab}_i$ are the conditions that a symmetry of the $\lambda$-bracket must satisfy in order to be Hamiltonian. Indeed, they guarantee that a solution (i.e., a Hamiltonian) exists for a generic vector field expressed in terms of the same coefficients. The compatibility conditions may or may not have the same solution as the conditions for a vector field to be a symmetry. Of course, all the solutions of the compatibility conditions are symmetries: they are components of a Hamiltonian vector field. The converse is in general not true, namely the solutions of the symmetry conditions may not be solutions of the compatibility ones. That would mean that there exist non Hamiltonian symmetries.

The compatibility conditions among the parameters in the LHS of the system \eqref{eq:HamVF_LP} can be found using the tools of \texttt{Janet} package. We compute the Janet basis for them, getting exactly the set of equations \eqref{eq:1stOrderSymLPdiff}. That means that all the first order symmetries are Hamiltonian vector fields.

We have proved the following theorem:
\begin{thm}\label{thm:1stCoho}
The first cohomology groups of $\{\cdot_{\mlambda}\cdot\}_1$ and $\{\cdot_{\mlambda}\cdot\}_2$ are not trivial. In particular, their first order components are isomorphic to $\R^2$.
The first order component of the first cohomology group for the Poisson Vertex Algebra  $(\hat{\A},\{\cdot_{\mlambda}\cdot\}_{LP})$ is trivial.
\end{thm}

\subsection{Deformations of Poisson brackets of hydrodynamic type}\label{ssec:defoLP}
\begin{deff}A \emph{$n$-th order deformation} of a PVA $(\hat{\A},\{\cdot_{\mlambda}\cdot\}_0)$ is a PVA defined by a deformed $\lambda$-bracket
\begin{equation}\label{eq:defoLambda_def}
 \{\cdot_{\mlambda}\cdot\}=\{\cdot_{\mlambda}\cdot\}_0+\sum_{k=1}^n\epsilon^k\{\cdot_{\mlambda}\cdot\}_{[k]}
\end{equation}
such that $\{\cdot_{\mlambda}\cdot\}$ is PVA-skewsymmetric and the PVA-Jacobi identity holds up to order $n$, namely
\begin{equation*}
 \{f_{\mlambda}\{g_{\mmu}h\}\}-\{g_{\mmu}\{f_{\mlambda}h\}\}-\{\{f_{\mlambda}g\}_{\mlambda+\mmu}h\}=O(\epsilon^{n+1}).
\end{equation*}
\end{deff}

A general Miura type transformation is a change of coordinates in the space of generators of the PVA. Using the grading introduced in \ref{ssec:PVACohom} we can define the \emph{Miura group} as the group of transformations of form
\begin{equation}\label{eq:Miura}
 \begin{split}
  &u^i\mapsto \tilde{u}^i=\sum_{k=0}^\infty\epsilon^kF^i_{[k]}(u;u_I)\quad|I|\leq k\\
&F^i_{[k]}\in\hat{\A},\quad\deg F^i_{[k]}=k\\
&\det\left(\frac{\dev F^i_{[0]}(p)}{\dev u^j}\right)\neq0.
 \end{split}
\end{equation}
\begin{deff}
A deformation is said to be \emph{trivial} if there exists an element $\phi_\epsilon$ of the group \eqref{eq:Miura} which pulls back $\{\cdot_{\mlambda}\cdot\}$ to $\{\cdot_{\mlambda}\cdot\}_0$,
\begin{equation*}
 \{\phi_\epsilon(a)_{\mlambda}\phi_\epsilon(b)\}_0=\phi_\epsilon\left(\{a_{\mlambda}b\}\right),\qquad\forall a,b\in\hat{\A}.
\end{equation*}
In terms of PVA cohomology, a deformed bracket is trivial if it is a coboundary in $\Omega^2(\hat{\A},\{\cdot_{\mlambda}\cdot\}_0)$.
\end{deff}
A first order deformation of \eqref{eq:defP1}, \eqref{eq:defP2} or \eqref{eq:LPPVA} is a second degree homogeneous bracket. In general, such a bracket is of the form
\begin{multline}\label{eq:LPPVADefo}
\left\{{p_i}_{\mlambda}{p_j}\right\}_{[1]}=A^{ab}_{ij}(p)\lambda_a\lambda_b+B^{a,bl}_{ij}(p)\dev_bp_l\lambda_a+\\+C^{al,bm}_{ij}(p)\dev_ap_l\dev_bp_m+D^{ab,l}(p)\dev_{ab}p_l
\end{multline}
in which each index can take values between $1$ and $d$ and we adopt the Einstein convention for the sum over repeated indices; moreover, the commas in the upper indices are inserted just for the convenience of the reader, namely to distinguish the different symmetry properties of the indices. Here, $A$, $B$, $C$ and $D$ are arbitrary functions of the $p$'s only. It should be apparent from the definition that $A^{ab}_{ij}$ and $D^{ab,l}_{ij}$ are symmetric in the exchange of $a$ and $b$ while $C^{al,bm}_{ij}$ must be symmetric in the simultaneous exchange of $(a,l)$ with $(b,m)$. The deformation depends on 108 parameters for $d=n=2$. The formalism of the Poisson Vertex Algebras makes finding the conditions on $A$, $B$, $C$ and $D$ for the bracket $\{\cdot_{\mlambda}\cdot\}_{[1]}$ to the first order deformation of \eqref{eq:defP1}, \eqref{eq:defP2}, and \eqref{eq:LPPVA} relatively simple, and anyhow straightforward. We will prove the following
\begin{thm}\label{thm:Trivia}
The first order second cohomology group for the Poisson Vertex Algebra associated to a multidimensional Poisson bracket of hydrodynamic type of form \eqref{eq:defP1}, \eqref{eq:defP2}, or \eqref{eq:LPPVA} for $d=n=2$ is trivial.
\end{thm}
\subsection{Proof of the Theorem \ref{thm:Trivia}}
In order to prove the theorem, we will proceed along two paths that are ultimately going to meet. First, we impose to the deformation  \eqref{eq:LPPVADefo} the constraints needed to get a first order deformed bracket, namely the skewsymmetry and the fulfillment of the PVA--Jacobi identity up to order $\epsilon$. Then, we shift to consider the trivial deformations of \eqref{eq:defP1}, \eqref{eq:defP2} and \eqref{eq:LPPVA}, namely the ones which are obtained by a Miura transformation of the bracket itself. We will show that the compatibility conditions of the latter ones and the reduced system of the former coincide. That means that all the compatible deformations of order 1 are trivial. While the condition of skewsymmetry is independent from the particular form of the undeformed bracket, both the trivial deformations and the PVA--Jacobi identities must be computed for each undeformed bracket. Hence, the full proof of the theorem is split in several lemmas.
\begin{lemma}\label{thm:defoskew}
A first order deformation of the $\lambda$-bracket of a PVA for $d=n$ is skewsymmetric if and only if the following conditions hold:
\begin{subequations}\label{eq:defoskew}
\begin{gather}\label{eq:1ODS1}
A^{ab}_{ij}=-A^{ab}_{ji}\\ \label{eq:1ODS2}
\frac{\dev A^{ab}_{ij}}{\dev p_l}=\frac{1}{2}\left(B^{a,bl}_{ij}-B^{a,bl}_{ji}\right)=\frac{1}{2}\left(B^{b,al}_{ij}-B^{b,al}_{ji}\right)\\ \label{eq:1ODS3}
B^{a,bl}_{ij}+B^{b,al}_{ji}=B^{b,al}_{ij}+B^{a,bl}_{ji}=2D^{ab,l}_{ij}+2D^{ab,l}_{ji}\\ \label{eq:1ODS4}
\frac{\dev B^{a,bm}_{ij}}{\dev p_l}+\frac{\dev B^{b,al}_{ji}}{\dev p_m}=\frac{\dev B^{b,al}_{ij}}{\dev p_m}+\frac{\dev B^{a,bm}_{ji}}{\dev p_l}=2C^{al,bm}_{ij}+2C^{al,bm}_{ji}
\end{gather}
\end{subequations}
\end{lemma}
\begin{proof}
We compute $\{{p_i}_{\mlambda}{p_j}\}_{[1]}+{}_\to\{{p_j}_{-\mdev-\mlambda}{p_i}\}_{[1]}$ and set equal to zero respectively the coefficients of $\lambda_a\lambda_b$, $\lambda_a\dev_bp_l$, $\dev_ap_l\dev_bp_m$ and $\dev_{ab}p_l$.
\end{proof}
In particular, for the case $d=n=2$, the condition of skewsymmetry is equivalent to impose the following form for the parameters of \eqref{eq:LPPVADefo}:
\begin{subequations}\label{eq:S-APar}
\begin{align}\label{eq:AS}
 A^{ab}_{ij(S)}&=\frac{1}{2}\left(A^{ab}_{ij}+A^{ab}_{ji}\right)\stackrel{\text{\eqref{eq:1ODS1}}}{=}0\\\label{eq:AA}
A^{ab}_{12(A)}&=\frac{1}{2}\left(A^{ab}_{12}-A^{ab}_{21}\right)=\tilde{A}^{ab}\\ \label{eq:BS}
B^{a,bl}_{ij(S)}&=\frac{1}{2}\left(B^{a,bl}_{ij}+B^{ab,l}_{ji}\right)=\tilde{B}^{a,bl}_{ij}\\ \label{eq:BA}
B^{a,bl}_{12(A)}&=\frac{1}{2}\left(B^{a,bl}_{12}-B^{ab,l}_{21}\right)\stackrel{\text{\eqref{eq:1ODS2}}}{=}\frac{\dev \tilde{A}^{ab}}{\dev p_l}\\ \label{eq:CS}
C^{al,bm}_{ij(S)}&=\frac{1}{2}\left(C^{al,bm}_{ij}+C^{al,bm}_{ji}\right)\\ \notag
&\stackrel{\text{\eqref{eq:1ODS4}}}{=}\frac{1}{4}\left(\frac{\dev \left(\tilde{B}^{a,bm}_{ij}+B^{a,bm}_{ij(A)}\right)}{\dev p_l}+\frac{\dev\left( \tilde{B}^{b,al}_{ij}-B^{b,al}_{ij(A)}\right)}{\dev p_m}\right)
\end{align}
\begin{align}\label{eq:CA}
C^{al,bm}_{12(A)}&=\frac{1}{2}\left(C^{al,bm}_{12}-C^{al,bm}_{21}\right)=\tilde{C}^{al,bm}\\ \label{eq:DS}
D^{ab,l}_{ij(S)}&=\frac{1}{2}\left(D^{ab,l}_{ij}+D^{ab,l}_{ji}\right)\\ \notag
&\stackrel{\text{\eqref{eq:1ODS3}}}{=}\frac{1}{4}\left(\tilde{B}^{a,bl}_{ij}+B^{a,bl}_{ij(A)}+\tilde{B}^{b,al}_{ij}-B^{b,al}_{ij(A)}\right)\\ \label{eq:DA}
D^{ab,l}_{12(A)}&=\frac{1}{2}\left(D^{ab,l}_{12}-D^{ab,l}_{21}\right)=\tilde{D}^{ab,l}.
\end{align}
\end{subequations}
Imposing the skewsymmetry condition reduces the number of free parameters (now they are the functions denoted with the tilde) to 43.
\begin{lemma}\label{thm:1stOrderDefoP1}
A homogeneous $\lambda$-bracket of degree 2 of form \eqref{eq:LPPVADefo} is a first order deformation of the bracket \eqref{eq:defP1} if and only if the following conditions hold:
\begin{subequations}\label{eq:1stOrderDefoP1}
\begin{gather}
 \sum_{\sigma(a,b,c)}\left[D^{ab,i}_{jk}\delta^c_i-D^{ab,k}_{ij}\delta^c_k+\frac{1}{2}\left(B^{a,bk}_{ij}+B^{a,bk}_{ji}\right)\delta^c_k\right]=0\\
 \notag B^{c,bi}_{jk}\delta^a_i+B^{c,ai}_{jk}\delta^b_i-\left(B^{a,bj}_{ik}-B^{a,bj}_{ki}\right)\delta^c_j+\left(B^{a,ck}_{ij}-2D^{ac,k}_{ij}\right)\delta^b_k+\\
 +\left(B^{b,ck}_{ij}-2D^{bc,k}_{ij}\right)\delta^a_k+2D^{ab,k}_{ji}\delta^c_k=0\\
 \sum_{\sigma(a,b,c)}\left(2C^{ak,bl}_{ij}-\frac{\dev D^{ab,l}_{ij}}{\dev p_k}-\dev{\dev D^{ab,k}_{ij}}{\dev p_l}\right)\delta^c_k=0\\ \notag
 2C^{ai,cl}_{jk}\delta^b_i+2C^{bi,cl}_{jk}\delta^a_i+\left(\frac{\dev B^{a,ck}_{ij}}{\dev p_l}-\frac{\dev B^{a,cl}_{ij}}{\dev p_k}-2\frac{\dev D^{ac,k}_{ij}}{\dev p_l}+2C^{ak,cl}_{ij}\right)\delta^b_k+\\
 +\left(\frac{\dev B^{b,ck}_{ij}}{\dev p_l}-\frac{\dev B^{b,cl}_{ij}}{\dev p_k}-2\frac{\dev D^{bc,k}_{ij}}{\dev p_l}+2C^{bk,cl}_{ij}\right)\delta^a_k+2\frac{\dev D^{ab,k}_{ji}}{\dev p_l}\delta^c_k=0\\ \notag
 \frac{\dev B^{b,cl}_{jk}}{\dev p_i}\delta^a_i-\frac{\dev B^{a,cl}_{ik}}{\dev p_j}\delta^b_j+2\left(C^{bk,cl}_{ij}-\frac{\dev D^{bc,k}_{ij}}{\dev p_l}\right)\delta^a_k-2\left(C^{ak,cl}_{ji}-\frac{\dev D^{ac,k}_{ji}}{\dev p_l}\right)\delta^b_k+\\
 +\frac{\dev}{\dev p_l}\left(B^{ab,k}_{ij}-2D^{ab,k}_{ij}\right)\delta^c_k=0\\ \notag
 2\frac{\dev D^{bc,l}_{jk}}{\dev p_i}\delta^a_i+\left(\frac{\dev B^{a,bk}_{ij}}{\dev p_l}-\frac{\dev B^{a,bl}_{ij}}{\dev p_k}+2C^{ak,bl}_{ij}-2\frac{\dev D^{ab,k}_{ij}}{\dev p_l}\right)\delta^c_k+\\ \notag
 +\left(\frac{\dev B^{a,ck}_{ij}}{\dev p_l}-\frac{\dev B^{a,cl}_{ij}}{\dev p_k}+2C^{ak,cl}_{ij}-2\frac{\dev D^{ac,k}_{ij}}{\dev p_l}\right)\delta^b_k+\\+2\left(C^{bk,cl}_{ij}+C^{ck,bl}_{ij}-\frac{\dev D^{bc,l}_{ij}}{\dev p_k}-\frac{\dev D^{bc,k}_{ij}}{\dev p_l}\right)\delta^a_k=0\\
 \sum_{\sigma(al,bm,cn)}\left(2\frac{\dev^2 C^{ak,bm}_{ij}}{\dev p_l\dev p_n}-\frac{\dev^2 C^{al,bm}_{ij}}{\dev p_k\dev p_n}-\frac{\dev^3D^{ab,k}_{ij}}{\dev p_l\dev p_m\dev p_n}\right)\delta^c_k=0 
\end{gather}
\begin{gather}
 \notag
 2\frac{\dev C^{bl,cm}_{ij}}{\dev p_i}\delta^a_i+\left(2\frac{\dev C^{bl,ck}_{ij}}{\dev p_m}+\frac{\dev C^{cm,bk}_{ij}}{\dev p_l}-2\frac{\dev^2D^{bc,k}_{ij}}{\dev p_l\dev p_m}-2\frac{\dev C^{bl,cm}_{ij}}{\dev p_k}\right)\delta^a_k+\\ \notag
 +\left(\frac{\dev^2B^{a,bk}_{ij}}{\dev p_l\dev p_m}-\frac{\dev^2B^{a,bl}_{ij}}{\dev p_k\dev p_m}+2\frac{\dev C^{ak,bl}_{ij}}{\dev p_m}-2\frac{\dev^2D^{ab,k}_{ij}}{\dev p_l\dev p_m}\right)\delta^c_k+\\
 +\left(\frac{\dev^2B^{a,ck}_{ij}}{\dev p_l\dev p_m}-\frac{\dev^2B^{a,cl}_{ij}}{\dev p_k\dev p_m}+2\frac{\dev C^{ak,cl}_{ij}}{\dev p_m}-2\frac{\dev^2D^{ac,k}_{ij}}{\dev p_l\dev p_m}\right)\delta^b_k=0\\ \notag
 \left(\frac{\dev C^{ak,bl}_{ij}}{\dev p_m}-\frac{\dev^2 D^{ab,l}_{ij}}{\dev p_m\dev p_k}+\frac{\dev C^{bk,al}_{ij}}{\dev p_m}-\frac{\dev^2 D^{ab,k}_{ij}}{\dev p_m\dev p_l}\right)\delta^c_k+\\ \notag
+\left(\frac{\dev C^{ck,al}_{ij}}{\dev p_m}+\frac{\dev C^{ak,cm}_{ij}}{\dev p_l}-\frac{\dev C^{al,cm}_{ij}}{\dev p_k}-\frac{\dev^2 D^{ac,k}_{ij}}{\dev p_m\dev p_l}\right)\delta^b_k+\\
+\left(\frac{\dev C^{ck,bl}_{ij}}{\dev p_m}+\frac{\dev C^{bk,cm}_{ij}}{\dev p_l}-\frac{\dev C^{bl,cm}_{ij}}{\dev p_k}-\frac{\dev^2 D^{bc,k}_{ij}}{\dev p_m\dev p_l}\right)\delta^a_k=0
\end{gather}
\end{subequations}
Repeated indices are summated according to Einstein's rule; $\sum_{\sigma(a,b,c)}$ means the complete symmetrization with respect to the listed indices (or couples of indices).
\end{lemma}
\begin{proof}
When computing the PVA-Jacobi identity for $\{\cdot_{\mlambda}\cdot\}$, we end up with a degree 0 term in $\epsilon$ which is the PVA-Jacobi identity for the undeformed bracket $\{\cdot_{\mlambda}\cdot\}_1$, plus a degree 1 term which reads
\begin{multline}\label{eq:1ODPVAJac}
 \epsilon\bigg(\left\{{p_i}_{\mlambda}\left\{{p_j}_{\mmu}{p_k}\right\}_1\right\}_{[1]}+\left\{{p_i}_{\mlambda}\left\{{p_j}_{\mmu}{p_k}\right\}_{[1]}\right\}_1+\\
-\left\{{p_j}_{\mmu}\left\{{p_i}_{\mlambda}{p_k}\right\}_1\right\}_{[1]}-\left\{{p_j}_{\mmu}\left\{{p_i}_{\mlambda}{p_k}\right\}_{[1]}\right\}_1+\\
-\left\{{\left\{{p_i}_{\mlambda}{p_j}\right\}_{[1]}}_{\mlambda+\mmu}{p_k}\right\}_1-\left\{{\left\{{p_i}_{\mlambda}{p_j}\right\}_1}_{\mlambda+\mmu}{p_k}\right\}_{[1]}\bigg)
\end{multline}
and terms of higher order that are discharged. The sets of equations \eqref{eq:1stOrderDefoP1} are then obtained collecting the homogeneous terms in $\lambda$, $\mu$ and derivatives of $p$, up to the third degree.
\end{proof}
We apply to \eqref{eq:1stOrderDefoP1} the skewsymmetry conditions \eqref{eq:S-APar}; all the algebraic relations, that can be found by direct inspection, among the 43 parameters are given in Appendix \ref{app:DefoPar1}. There are still 9 functions (according to our choice, $\tilde{A}^{11}$, $\tilde{A}^{12}$, $\tilde{A}^{22}$, $\tilde{B}^{1,12}_{11}$, $\tilde{B}^{1,22}_{11}$, $\tilde{B}^{2,11}_{11}$, $\tilde{B}^{2,21}_{22}$, $\tilde{B}^{2,11}_{22}$, and $\tilde{B}^{1,22}_{22}$) which must satisfy the following set of linear PDEs in order to be the components of the first order deformed bracket.
\begin{subequations}\label{eq:PVAJ_P1}
\begin{align}
 \frac{\dev \tilde{B}^{2,11}_{11}}{\dev q}&=\frac{\dev\tilde{B}^{1,22}_{11}}{\dev p}+2\frac{\dev^2\tilde{A}^{22}}{\dev p^2}\\
 \frac{\dev \tilde{B}^{1,22}_{22}}{\dev p}&=\frac{\dev\tilde{B}^{2,11}_{22}}{\dev q}-2\frac{\dev^2\tilde{A}^{11}}{\dev q^2}.
\end{align}
\end{subequations}
The same procedure can be repeated for the deformations of \eqref{eq:defP2}. In this case we do not start looking for the general set of conditions for any $d$, but compute explictly the PVA--Jacobi identity at the first order \eqref{eq:1ODPVAJac} for $\{\cdot_{\mlambda}\cdot\}_2$ and $\{\cdot_{\mlambda}\cdot\}_{[1]}$, having imposed the form \eqref{eq:S-APar} to the parameters of the deformation. All the 43 parameters can be expressed as linear combinations and derivatives of just 9 of them, namely $\tilde{A}^{11}$, $\tilde{A}^{12}$, $\tilde{A}^{22}$, $\tilde{B}^{1,11}_{11}$, $\tilde{B}^{1,21}_{11}$, $\tilde{B}^{1,12}_{22}$, $\tilde{B}^{2,11}_{22}$, $\tilde{B}^{2,12}_{22}$, and $\tilde{B}^{2,21}_{22}$. The formulas for the remaining can be found solving linear algebraic equations and are explicitly given in Appendix \ref{app:DefoPar2}. The nine parameters we are left with must satisfy the following pair of linear PDEs:
\begin{subequations}\label{eq:PVAJ_P2}
\begin{align}
 \frac{\dev\tilde{B}^{1,12}_{22}}{\dev q}+\frac{\dev\tilde{B}^{2,11}_{22}}{\dev q}&=\frac{\dev \tilde{B}^{2,12}_{22}}{\dev p}\\
 \frac{\dev \tilde{B}^{1,11}_{11}}{\dev q}+2\frac{\dev^2\tilde{A}^{11}}{\dev q^2}+\frac{\dev^2\tilde{A}^{22}}{\dev p^2}&=2\frac{\dev\tilde{B}^{2,12}_{22}}{\dev q}+\frac{\dev\tilde{B}^{2,21}_{22}}{\dev q}+\frac{\dev\tilde{B}^{1,21}_{11}}{\dev p}+4\frac{\dev^2\tilde{A}^{12}}{\dev p\dev q}.
\end{align}
\end{subequations}
Finally, we perform the same computation of Lemma \ref{thm:1stOrderDefoP1} for the third class of $\lambda$-brackets.
\begin{lemma}\label{thm:1stOrderDefoLP}
 A homogeneous $\lambda$-bracket of degree 2 of form \eqref{eq:LPPVADefo} is a first order deformation of the bracket \eqref{eq:LPPVA} if and only if the following conditions hold:
\begin{subequations}\label{eq:1stOrderDefo}
\begin{gather} \label{eq:PVAlll}
 D^{ab,c}_{ji}p_k+D^{ab,c}_{jk}p_i+\left(D^{ab,l}_{ji}\delta^c_k+D^{ab,l}_{jk}\delta^c_i\right)p_l+\pc (a,b,c)=0\\ \notag
2\left(A^{bc}_{ik}\delta^a_j+A^{ac}_{ik}\delta^b_j\right)+2\left(A^{ab}_{ji}\delta^c_k-A^{ab}_{ik}\delta^c_j+A^{ab}_{kj}\delta^c_i\right)+\\ \notag
-\left(B^{a,bc}_{ki}-B^{a,bc}_{ik}\right)p_j-\left(B^{c,ab}_{jk}-B^{c,ba}_{jk}\right)p_i+\\ \label{eq:PVAllm}
-\left[\left(B^{a,bl}_{ki}-B^{a,bl}_{ik}\right)\delta^c_j+B^{c,bl}_{jk}\delta^a_i+B^{c,al}_{jk}\delta^b_i\right]p_l+\\ \notag
-\left[\left(2D^{bc,l}_{ji}-B^{c,bl}_{ji}\right)\delta^a_k+\left(2D^{ca,l}_{ji}-B^{c,al}_{ji}\right)\delta^b_k+2D^{ab,l}_{ji}\delta^c_k\right]p_l+\\ \notag
-\left(2D^{cb,a}_{ji}-B^{c,ba}_{ji}+2D^{ca,b}_{ji}-B^{c,ab}_{ji}+2D^{ab,c}_{ji}\right)p_k=0\\ \label{eq:PVAd3}
\sum_{\sigma(a,b,c)}\left[\left(\frac{\dev D^{ab,l}_{ij}}{\dev p_m}+\frac{\dev D^{ab,m}_{ij}}{\dev p_l}-2C^{am,bl}_{ij}\right)p_m\delta^c_k\right.+\\ \notag+\left.\left(\frac{\dev D^{ab,l}_{ij}}{\dev p_c}+\frac{\dev D^{ab,c}_{ij}}{\dev p_l}-2C^{ac,bl}_{ij}\right)\right]=0\\ \notag
\left(C^{ba,cl}_{ji}+C^{ab,cl}_{ji}\right)p_k-\left(C^{ba,cl}_{jk}+C^{ab,cl}_{jk}\right)p_i-\frac{\dev}{\dev p_l}\left(D^{ab,c}_{ji}+D^{bc,a}_{ji}+D^{ca,b}_{ji}\right)p_k+\\ \notag
-\left(D^{bc,a}_{jk}+D^{ac,b}_{jk}\right)\delta^l_i-D^{ab,c}_{ji}\delta^l_k-\left(D^{bc,l}_{jk}\delta^a_i+D^{ca,l}_{jk}\delta^b_i+D^{ab,l}_{jk}\delta^c_i\right)+\\ \label{eq:PVAlld}
+\left[\left(C^{bm,cl}_{ji}\delta^a_k+C^{am,cl}_{ji}\delta^b_k\right)-\left(C^{bm,cl}_{jk}\delta^a_i+C^{am,cl}_{jk}\delta^b_i\right)+\right.\\ \notag
\left.-\frac{\dev}{\dev p_l}\left(D^{ab,m}_{ji}\delta^c_k+D^{bc,m}_{ji}\delta^a_k+D^{ca,m}_{ji}\delta^b_k\right)\right]p_m=0
\end{gather}
\begin{gather}\label{eq:PVAlmd}
B^{b,cl}_{ik}\delta^a_j-B^{a,cl}_{jk}\delta^b_i+B^{b,cl}\delta^a_k-B^{a,cl}_{ij}\delta^b_k+B^{b,al}_{jk}\delta^c_i+\\ \notag
+B^{a,cb}_{ik}\delta^l_j-B^{b,ca}_{jk}\delta^l_i+\left(2D^{ab,c}_{ij}-B^{a,bc}_{ij}\right)+
+\frac{\dev B^{a,cl}_{ik}}{\dev p_b}p_j-\frac{\dev B^{b,cl}_{jk}}{\dev p_a}p_i+\\ \notag
+\left[2C^{ab,cl}_{ji}-2C^{ba,cl}_{ij}+2\frac{\dev D^{cb,a}_{ij}}{\dev p_l}-2\frac{\dev D^{ca,b}_{ji}}{\dev p_l}+\frac{\dev}{\dev p_l}\left(2D^{ab,c}_{ij}-B^{a,bc}_{ij}\right)\right]p_k+\\ \notag
+\left[\frac{\dev B^{a,cl}_{ik}}{\dev p_m}\delta^b_j-\frac{\dev B^{b,cl}_{jk}}{\dev p_m}\delta^a_i+2\left(\frac{\dev D^{cb,m}_{ij}}{\dev p_l}-C^{bm,cl}_{ij}\right)\delta^a_k+\right.\\ \notag\left. -2\left(\frac{\dev D^{ab,m}_{ji}}{\dev p_l}-C^{am,cl}_{ji}\right)\delta^b_k+\frac{\dev}{\dev p_l}\left(2D^{ab,m}_{ij}-B^{a,bm}_{ij}\right)\delta^c_k\right]=0\\ \notag
\left(C^{ac,bl}_{ji}+C^{ab,cl}_{ji}-\frac{\dev D^{ab,c}_{ji}}{\dev p_l}-\frac{\dev D^{ac,b}_{ji}}{\dev p_l}+\frac{\dev D^{bc,a}_{ij}}{\dev p_l}+\frac{\dev D^{bc,l}_{ij}}{\dev p_a}-C^{ba,cl}_{ij}-C^{ca,bl}_{ij}\right)p_k+\\ \label{eq:PVAld2}
-\frac{\dev D^{bc,l}_{jk}}{\dev p_a}p_i-D^{bc,a}_{jk}\delta^l_i-D^{ab,l}_{jk}\delta^c_i-D^{ac,l}_{jk}\delta^b_i+D^{bc,l}_{ji}\delta^a_k+\\ \notag
+\left[-\frac{\dev D^{bc,l}_{jk}}{p_m}\delta^a_i+\left(\frac{\dev D^{bc,l}_{ij}}{\dev p_m}-\frac{\dev D^{bc,m}_{ij}}{\dev p_l}-C^{bm,cl}_{ij}-C^{cm,bl}_{ij}\right)\delta^a_k+\right.\\ \notag
\left.+\left(C^{am,bl}_{ji}-\frac{\dev D^{ab,m}_{ji}}{\dev p_l}\right)\delta^c_k+\left(C^{am,cl}_{ji}-\frac{\dev D^{ac,m}_{ji}}{\dev p_l}\right)\delta^b_k\right]p_m=0\\ \notag
-2\frac{\dev}{\dev p_a}\left(C^{bl,cm}_{jk}p_i\right)+2\frac{\dev C^{bl,cm}_{ij}}{\dev p_a}p_k+2C^{bl,cm}_{ji}\delta^a_k\\ \notag
+\frac{\dev}{\dev p_m}\left(2C^{ac,bl}_{ji}-2C^{ca,bl}_{ij}+2\frac{\dev D^{bc,a}_{ij}}{\dev p_l}-2\frac{\dev D^{ab,c}_{ji}}{\dev p_l}\right)p_k+\\ \label{eq:PVAldd}
+\frac{\dev}{\dev p_l}\left(2C^{ab,cm}_{ji}-2C^{ba,cm}_{ij}+2\frac{\dev D^{bc,a}_{ij}}{\dev p_m}-2\frac{\dev D^{ac,b}_{ji}}{\dev p_m}\right)p_k+\\ \notag
-2C^{bl,am}_{jk}\delta^c_i-2C^{cm,al}_{jk}\delta^b_i-2C^{bl,ca}_{jk}\delta^m_i-2C^{cm,ba}_{jk}\delta^l_i+\\ \notag
+2\left(C^{ac,bl}_{ji}-\frac{\dev D^{ab,c}_{ji}}{\dev p_l}\right)\delta^m_k+2\left(C^{ab,cm}_{ji}-\frac{\dev D^{ac,b}_{ji}}{\dev p_m}\right)\delta^l_k=0\\ \label{eq:PVAddd}
\sum_{\sigma(al,bm,cn)}\left[\frac{\dev}{\dev p_n}\left(\frac{\dev C^{al,bm}_{ij}}{\dev p_c}p_k-2\frac{\dev C^{ac,bm}_{ij}}{\dev p_l}p_k+\frac{\dev^2 D^{ab,c}_{ij}}{\dev p_m\dev p_l}p_k\right)+\right.\\ \notag
\left.+p_s\frac{\dev}{\dev p_n}\left(\frac{\dev C^{al,bm}_{ij}}{\dev p_s}-2\frac{\dev C^{as,bm}_{ij}}{\dev p_l}+\frac{\dev^2 D^{ab,s}_{ij}}{\dev p_m\dev p_l}\right)\delta^c_k\right]=0
\end{gather}
\begin{multline}\label{eq:PVAd2d}
 \frac{\dev}{\dev p_m}\left[\left(\frac{\dev D^{ab,l}_{ij}}{\dev p_c}+\frac{\dev D^{ab,c}_{ij}}{\dev p_l}-C^{ac,bl}_{ij}-C^{bc,al}_{ij}\right)p_k\right]+\\
+\left(\frac{\dev C^{al,cm}_{ij}}{\dev p_b}+\frac{\dev C^{bl,cm}_{ij}}{\dev p_a}-\frac{\dev C^{ab,cm}_{ij}}{\dev p_l}-\frac{\dev C^{ba,cm}_{ij}}{\dev p_l}+\right.\\
\left.-\frac{\dev C^{al,cb}_{ij}}{\dev p_m}-\frac{\dev C^{bl,ca}_{ij}}{\dev p_m}+\frac{\dev^2 D^{ac,b}_{ij}}{\dev p_l\dev p_m}+\frac{\dev^2 D^{bc,a}_{ij}}{\dev p_l\dev p_m}\right)p_k+\\
+p_n\left[\left(\frac{\dev^2 D^{ab,l}_{ij}}{\dev p_m\dev p_n}+\frac{\dev^2 D^{ab,n}_{ij}}{\dev p_l\dev p_m}-\frac{\dev C^{an,bl}_{ij}}{\dev p_m}-\frac{\dev C^{bn,al}_{ij}}{\dev p_m}\right)\delta^c_k+\right.\\
\left.\left(\frac{\dev^2 D^{ac,n}_{ij}}{\dev p_l\dev p_m}+\frac{\dev C^{al,cm}_{ij}}{\dev p_n}-\frac{\dev C^{an,cm}_{ij}}{\dev p_l}-\frac{\dev C^{cn,al}_{ij}}{\dev p_m}\right)\delta^b_k+\right.\\
\left.\left(\frac{\dev^2 D^{ac,n}_{ij}}{\dev p_l\dev p_m}+\frac{\dev C^{al,cm}_{ij}}{\dev p_n}-\frac{\dev C^{an,cm}_{ij}}{\dev p_l}-\frac{\dev C^{cn,al}_{ij}}{\dev p_m}\right)\delta^b_k+\right.\\
\left.\left(\frac{\dev^2 D^{bc,n}_{ij}}{\dev p_l\dev p_m}+\frac{\dev C^{bl,cm}_{ij}}{\dev p_n}-\frac{\dev C^{bn,cm}_{ij}}{\dev p_l}-\frac{\dev C^{cn,bl}_{ij}}{\dev p_m}\right)\delta^a_k\right]=0
\end{multline}
\end{subequations}
The notation $\pc (a,b,c)$ means cyclic permutations of the indices $(a,b,c)$.
\end{lemma}
\paragraph{Remark}Let us consider the trivial case $d=1$. The undeformed bracket reads $\{p_{\lambda}p\}_{LP}=-2p\lambda-p'$ (the prime means the only derivative of $p$, namely wrt $x$), which is the so-called \emph{Virasoro-Magri PVA} with central charge 0 (see Ex. 1.18 in \cite{BdSK09}). We easily get the well known result, shown for instance in \cite{dGMS05}, that such deformations do not exist in the scalar case. From the skewsymmetry conditions we get (now the indices have become useless) $A=0$, $2D=B$, $2C=B'$; moreover, \eqref{eq:PVAlll} is enough to get $D=0$, hence $B=C=0$.
\paragraph*{}
We follow the same approach as for the bracket $\{\cdot_{\mlambda}\cdot\}_1$, setting $d=2$ and expressing the parameters of the deformation according to \eqref{eq:S-APar}. In order to find the algebraic relations among the 43 parameters of the deformation we use the Mathematica package \texttt{SYM} \cite{DT05}. We get an overdetermined system of 45 equations for 9 unknown functions $\tilde{A}^{11}$, $\tilde{A}^{12}$, $\tilde{A}^{22}$, $\tilde{B}^{1,22}_{11}$, $\tilde{B}^{2,11}_{11}$, $\tilde{B}^{1,21}_{11}$, $\tilde{B}^{2,11}_{22}$, $\tilde{B}^{1,22}_{22}$, and $\tilde{B}^{2,12}_{22}$. The expressions for the remaining ones in terms of these nine are left to the Appendix \ref{app:DefoParLP}. The Janet basis of the system of PDEs according which the deformed $\lambda$-bracket is a PVA up to the first order are the following two equations
\begin{subequations}\label{eq:PVAJ_LP}
\begin{multline}\label{eq:condPVA1}
-\frac{3}{4}p^2\tilde{A}^{11}+\frac{3}{4}q^2\tilde{A}^{22}-\frac{3}{4}q^3\tilde{B}^{1,22}_{11}-\frac{3}{4}p^3\tilde{B}^{2,11}_{22}-\frac{3}{8}q^3\frac{\dev\tilde{A}^{22}}{\dev q}+\frac{3}{8}p^3\frac{\dev\tilde{A}^{11}}{\dev p}+\\
-\frac{1}{4}p^2q^2\frac{\dev\tilde{B}^{2,11}_{11}}{\dev p}-\frac{1}{2}p^3q\frac{\dev\tilde{B}^{1,22}_{22}}{\dev p}-\frac{1}{2}pq^2\tilde{B}^{2,11}_{11}-\frac{1}{4}pq^2\tilde{B}^{1,21}_{11}-\frac{1}{4}p^2q\tilde{B}^{2,12}_{22}+\\
-\frac{1}{2}p^2q\tilde{B}^{1,22}_{22}-p^2q^2\frac{\dev^2\tilde{A}^{11}}{\dev q^2}-\frac{1}{2}p^3q\frac{\dev^2\tilde{A}^{11}}{\dev p\dev q}+p^2q^2\frac{\dev^2\tilde{A}^{22}}{\dev p^2}+\frac{1}{2}pq^3\frac{\dev^2\tilde{A}^{22}}{\dev p\dev q}+\\
+\frac{5}{8}p^2q\frac{\dev\tilde{A}^{11}}{\dev q}+\frac{1}{4}p^2q^2\frac{\dev\tilde{B}^{2,12}_{22}}{\dev q}-\frac{1}{4}pq^2\frac{\dev\tilde{A}^{12}}{\dev q}-\frac{1}{2}pq^3\frac{\dev\tilde{B}^{2,11}_{11}}{\dev q}+\frac{1}{2}p^3q\frac{\dev\tilde{B}^{2,11}_{22}}{\dev q}+\\+\frac{1}{2}pq^3\frac{\dev\tilde{B}^{1,22}_{11}}{\dev p}
+\frac{1}{4}p^2q\frac{\dev\tilde{A}^{12}}{\dev p}-\frac{1}{4}p^2q^2\frac{\dev\tilde{B}^{1,22}_{22}}{\dev q}+\frac{1}{4}p^2q^2\frac{\dev\tilde{B}^{1,21}_{11}}{\dev p}-\frac{5}{8}pq^2\frac{\dev\tilde{A}^{22}}{\dev p}=\\=0
\end{multline}
and
\begin{multline}\label{eq:condPVA2}
-5p^2q^2\frac{\dev^2\tilde{A}^{12}}{\dev p\dev q}-2p^3q\frac{\dev^2\tilde{A}^{12}}{\dev p^2}+6p^2q^2\frac{\dev^2\tilde{A}^{11}}{\dev q^2}+5p^3q\frac{\dev^2\tilde{A}^{11}}{\dev p\dev q}+2pq^3\frac{\dev^2\tilde{A}^{22}}{\dev p \dev q}+\\
+\frac{11}{4}p^2\tilde{A}^{11}-\frac{7}{4}q^2\tilde{A}^{22}+\frac{15}{4}q^3\tilde{B}^{1,22}_{11}+\frac{3}{4}p^3\tilde{B}^{2,11}_{22}-\frac{1}{8}q^3\frac{\dev\tilde{A}^{22}}{\dev q}-\frac{19}{8}p^3\frac{\dev\tilde{A}^{11}}{\dev p}+\\
+p^4\frac{\dev^2\tilde{A}^{11}}{\dev p^2}+q^4\frac{\dev^2\tilde{A}^{22}}{\dev q^2}+\frac{5}{4}p^2q^2\frac{\dev\tilde{B}^{2,11}_{11}}{\dev p}+\frac{1}{2}p^3q\frac{\dev\tilde{B}^{1,22}_{22}}{\dev p}+\frac{5}{2}pq^2\tilde{B}^{2,11}_{11}+\frac{9}{4}pq^2\tilde{B}^{1,21}_{11}+\\
-\frac{3}{4}p^2q\tilde{B}^{2,12}_{22}+\frac{1}{2}p^2q\tilde{B}^{1,22}_{22}+2q^4\frac{\dev\tilde{B}^{1,22}_{11}}{\dev q}+2pq^3\frac{\dev\tilde{B}^{1,21}_{11}}{\dev q}-2p^3q\frac{\dev\tilde{B}^{2,12}_{22}}{\dev p}-2p^4\frac{\dev\tilde{B}^{2,11}_{22}}{\dev p}+\\
-2pq^3\frac{\dev^2\tilde{A}^{12}}{\dev q^2}-\frac{17}{8}p^2q\frac{\dev\tilde{A}^{11}}{\dev q}-\frac{9}{4}p^2q^2\frac{\dev\tilde{B}^{2,12}_{22}}{\dev q}+\frac{1}{4}pq^2\frac{\dev\tilde{A}^{12}}{\dev q}+\frac{5}{2}pq^3\frac{\dev\tilde{B}^{2,11}_{11}}{\dev q}+\\
-\frac{7}{2}p^3q\frac{\dev\tilde{B}^{2,11}_{22}}{\dev q}+\frac{1}{2}pq^3\frac{\dev\tilde{B}^{1,22}_{11}}{\dev p}-\frac{5}{4}p^2q\frac{\dev\tilde{A}^{12}}{\dev p}+\frac{1}{4}p^2q^2\frac{\dev\tilde{B}^{1,22}_{22}}{\dev q}+\frac{3}{4}p^2q^2\frac{\dev\tilde{B}^{1,21}_{11}}{\dev p}+\\+
\frac{13}{8}pq^2\frac{\dev\tilde{A}^{22}}{\dev p}=0.
\end{multline}
\end{subequations}
\paragraph*{}
Now, let us consider the trivial deformations of \eqref{eq:defP1}, \eqref{eq:defP2}, and \eqref{eq:LPPVA}, namely the deformed brackets given by performing a general Miura transformation \eqref{eq:Miura} of the first order to the undeformed brackets. Such a change of coordinates will have the form
\begin{equation*}
 p_i\mapsto P_i=p_i+\sum_{j,k=1,2}\epsilon F^{jk}_i(p,q)\dev_jp_k
\end{equation*}
and thus depends on 8 arbitrary functions of $(p_1\equiv p,p_2\equiv q)$. We compute $\{{P_i}(p)_{\mlambda}{P_j}(p)\}_{1,2,LP}$, which is in all three cases very straightforward. We start by the expansion to the order $\epsilon$,
\begin{equation*}
 \{{P_i}_{\mlambda}{P_j}\}_{1,2,LP}=
  \{{p_i}_{\mlambda}{p_j}\}_{1,2,LP}+\epsilon\left(\left\{{F^{al}_i\dev_ap_l}_{\mlambda}{p_j}\right\}_{1,2,LP}+\left\{{p_i}_{\mlambda}{F^{al}_j\dev_ap_l}\right\}_{1,2,LP}\right)+O(\epsilon^2)
\end{equation*}
and then we use the master formula \eqref{eq:MasterFormula} for the two latter brackets. The expression we get is written in terms of the `old' coordinates; up to the first order, we can invert the transformation by $p_i=P_i-\epsilon F^{al}_i(P)\dev_aP_l$, getting the formula for the deformed bracket $\{{P_i(p)}_{\mlambda}{P_j(p)}\}_{1,2,LP}=\{{P_i}_{\mlambda}{P_j}\}_{1,2,LP}+\epsilon\{{P_i}_{\mlambda}{P_j}\}_{[1]}$. For convenience, we provide only the 9 parameters of each first order deformed bracket we have chosen to express all the other ones in terms of. For deformations of \eqref{eq:defP1} we have:
\begin{subequations}\label{eq:MiuraPVAP1}
\begin{align}
\tilde{A}^{11}&=F^{11}_2\\
2\tilde{A}^{12}&=F^{21}_2-F^{12}_1\\
\tilde{A}^{22}&=-F^{22}_1\\
B^{1,12}_{11}&=2\frac{\dev F^{12}_1}{\dev p}-2\frac{\dev F^{11}_1}{\dev q}\\
B^{1,22}_{11}&=2\frac{\dev F^{22}_1}{\dev p}-\frac{\dev F^{21}_1}{\dev q}\\
B^{2,11}_{11}&=-\frac{\dev F^{21}_1}{\dev p}\\
B^{2,21}_{22}&=2\frac{\dev F^{11}_2}{\dev q}-2\frac{\dev F^{22}_2}{\dev p}\\
B^{2,11}_{22}&=2\frac{\dev F^{11}_2}{\dev q}-\frac{\dev F^{12}_2}{\dev p}\\
B^{2,12}_{22}&=-\frac{\dev F^{12}_2}{\dev q}
\end{align}
\end{subequations}
Following the same procedure, the parameters we chose for the deformations of \eqref{eq:defP2} are
\begin{subequations}\label{eq:MiuraPVAP2}
\begin{align}
\tilde{A}^{11}&=-F^{11}_1+F^{12}_2\\
2\tilde{A}^{12}&=F^{22}_2-F^{21}_1-F^{12}_1\\
\tilde{A}^{22}&=-F^{22}_1\\
B^{1,11}_{11}&=2\frac{\dev F^{11}_1}{\dev q}-2\frac{\dev F^{12}_1}{\dev p}\\
B^{1,21}_{11}&=2\frac{\dev F^{21}_1}{\dev q}-\frac{\dev F^{22}_1}{\dev p}\\
B^{1,12}_{22}&=2\frac{\dev F^{12}_2}{\dev p}-2\frac{\dev F^{11}_2}{\dev q}\\
B^{2,11}_{22}&=2\frac{\dev F^{11}_2}{\dev q}-\frac{\dev F^{12}_2}{\dev p}-\frac{\dev F^{21}_2}{\dev p}\\
B^{2,12}_{22}&=\frac{\dev F^{12}_2}{\dev q}-\frac{\dev F^{21}_2}{\dev q}\\
B^{2,21}_{22}&=2\frac{\dev F^{21}_2}{\dev q}-2\frac{\dev F^{22}_2}{\dev p}
\end{align}
\end{subequations}
Finally, for the trivial deformations of \eqref{eq:LPPVA} we get
\begin{subequations}\label{eq:MiuraPVALP}
\begin{align}\label{eq:MPVA1}
\tilde{A}^{11}&=qF^{11}_1-2pF^{11}_2-qF^{12}_2\\ \label{eq:MPVA2}
2\tilde{A}^{12}&=pF^{11}_1+2q F^{12}_1-pF^{12}_2+q F^{21}_2-2p F^{21}_2-qF^{22}_2\\ \label{eq:MPVA3}
\tilde{A}^{22}&=pF^{21}_1+2q F^{22}_1-pF^{22}_2\\ \label{eq:MPVA4}
B^{1,21}_{11}&=F^{12}_1-2q\frac{\dev F^{21}_1}{\dev q}+p\frac{\dev F^{12}_1}{\dev p}-2p\frac{\dev F^{21}_1}{\dev p}+q\frac{\dev F^{22}_1}{\dev p}\\ \label{eq:MPVA5}
B^{12,2}_{11}&=F^{22}_1+p\frac{\dev F^{12}_1}{\dev q}+2p\frac{\dev F^{21}_1}{\dev q}-q\frac{\dev F^{22}_1}{\dev q}-4p\frac{\dev F^{22}_1}{\dev p}\\ \label{eq:MPVA6}
B^{2,11}_{11}&=-F^{12}_1-2p\frac{\dev F^{11}_1}{\dev q}+p\frac{\dev F^{12}_1}{\dev p}+2p\frac{\dev F^{21}_1}{\dev p}+q\frac{\dev F^{22}_1}{\dev p}\\ \label{eq:MPVA7}
B^{2,12}_{22}&= F^{21}_2+p\frac{\dev F^{11}_2}{\dev q}-2q\frac{\dev F^{12}_2}{\dev q}+q\frac{\dev F^{21}_2}{\dev q}-2p\frac{\dev F^{12}_2}{\dev p}\\ \label{eq:MPVA8}
B^{2,11}_{22}&=F^{11}_2+q\frac{\dev F^{21}_2}{\dev p}+2q\frac{\dev F^{12}_2}{\dev p}-p\frac{\dev F^{11}_2}{\dev p}-4q\frac{\dev F^{11}_2}{\dev q}\\ \label{eq:MPVA9}
B^{2,12}_{22}&=-F^{21}_2-2q\frac{\dev F^{22}_2}{\dev p}+q\frac{\dev F^{21}_2}{\dev q}+2q\frac{\dev F^{12}_2}{\dev q}+p\frac{\dev F^{11}_2}{\dev q}
\end{align}
\end{subequations}
In the three sets of equations \eqref{eq:MiuraPVAP1}, \eqref{eq:MiuraPVAP2}, and \eqref{eq:MiuraPVALP} we have dropped the tilde from the parameters $B$'s because, by definition \eqref{eq:S-APar}, we have $\tilde{B}^{a,bc}_{ii}=B^{a,bc}_{ii}$.
Since the three brackets we have just defined are the Miura transformed of the undeformed ones, they are a first order deformation of a PVA bracket; the sets of coefficients satisfy, as it can be easily checked, the corresponding PVA--Jacobi identities up to the first order.

We can regard each set of equations \eqref{eq:MiuraPVAP1}, \eqref{eq:MiuraPVAP2}, and \eqref{eq:MiuraPVALP} as an inhomogeneous linear system of 9 PDEs for the 8 unknown functions $F$'s. A solution of the system, if there exists, is the set of the eight parameters of a Miura transformation which produces a given coboundary.

The compatibility conditions for \eqref{eq:MiuraPVAP1} are \eqref{eq:PVAJ_P1}; the compatibility conditions for \eqref{eq:MiuraPVAP2} are \eqref{eq:PVAJ_P2}. Computing the compatibility conditions for \eqref{eq:MiuraPVALP} we get a system of two second order differential equations, whose Janet basis is exactly \eqref{eq:condPVA1} and \eqref{eq:condPVA2}.

That means that a generic first order cocycle, i.e. a first order deformed bracket, can be written in terms of the nine parameters if and only if they satisfy the corresponding pair of linear PDEs \eqref{eq:PVAJ_P1}, \eqref{eq:PVAJ_P2} or \eqref{eq:PVAJ_LP}. On the other hand, the same conditions allow to find the eight parameters of a Miura transformations for which we get that cocycle. It follows that every cocycle in $\Omega^2_{[1]}(\hat{\A},\{\cdot_{\mlambda}\cdot\})$ is a coboundary, so that $H^2_{[1]}(\hat{\A},\{\cdot_{\mlambda}\cdot\})=0$, for $(\hat{\A},\{\cdot_{\mlambda}\cdot\})$ a $2$-dimensional Poisson Vertex Algebra of hydrodynamic type with $2$ derivations.

\section{Concluding remarks}
In this paper we have formulated the theory of multidimensional Poisson Vertex Algebras, showing how it can be applied to the study of evolutionary Hamiltonian PDEs. In particular, we have proved a first result in the theory of deformations of multidimensional Poisson brackets of hydrodynamic type, namely the characterization of the first order first and second cohomology groups for the $d=n=2$ normal forms of the bracket, according to the classification by Ferapontov and collaborators \cite{FOS11}. 

We have proved that only the Lie-Poisson $\lambda$-bracket can have a trivial first cohomology group, while the second cohomology group is trivial for all the first order deformations of the normal forms of $\lambda$-brackets of hydrodynamic type.  We will devote further investigations to higher order deformations of the $\lambda$-brackets, aiming to characterize the full second cohomology group, in the spirit of \cite{G01}, \cite{dGMS05} and \cite{DZ}.
A deeper analysis of the classification of the normal forms of the Poisson brackets of hydrodynamic type seems to be an important task in order to provide general results. For example, we notice that a constant Poisson bracket with generating metrics $g^2=\lambda g^1$ cannot be reduced to either \eqref{eq:defP1} or \eqref{eq:defP2}. 
\appendix
\section{Proof of Lemma \ref{lemma:PVAJacobi}}\label{app:proofPVAJacobi}
Let us consider the three generators $u^i(x)$, $u^j(y)$ and $u^k(z)$. For convenience, we drop the boldface typesetting used in Section \ref{ssec:mPVA} to denote that the variables $x,y,z$ are coordinates in $\R^d$. Let us consider the double Fourier-like transform
\begin{equation}
\int e^{\mlambda\cdot(x-z)}e^{\mmu\cdot(y-z)}\{u^i(x),\{u^j(y),u^k(z)\}\}\ud x\ud y 
\end{equation}
 The first step is to expand the outer bracket , which gives
\begin{multline*}
\int e^{\mlambda\cdot(x-z)}e^{\mmu\cdot(y-z)}\left(\mdev_z^L\{u^i(x),u^l(z)\}\right)\frac{\dev \{u^j(y),u^k(z)\}}{\dev u^l_L}\ud x\ud y\\
=\int e^{\mlambda\cdot(x-z)}\frac{\dev}{\dev u^l_L}\left(e^{\mmu\cdot(y-z)}\{u^j(y),u^k(z)\}\right)\left(\mdev_z^L\{u^i(x),u^l(z)\}\right)\ud x\ud y.
\end{multline*}
If we perform the integration with respect to $y$, which appears only in the first parenthesis, we get by definition the $\lambda$-bracket (with parameter $\mmu$) of the two generators $u^j$ and $u^k$. Thus, we have got the partial result
\begin{multline*}
\int e^{\mlambda\cdot(x-z)}e^{\mmu\cdot(y-z)}\{u^i(x),\{u^j(y),u^k(z)\}\}\ud x\ud y\\=\int e^{\mlambda\cdot(x-z)}\{u^i(x),\{u^j_{\mmu}u^k\}(z)\}\ud x  .
\end{multline*}
Let us for simplicity denote $\{u^j_{\mmu}u^k\}(z)=g(z)$. A step backwards in the computation brings us back to
\begin{multline*}
\int e^{\mlambda\cdot(x-z)}\frac{\dev g(z)}{\dev u^l_L}\left(\mdev_z^L\{u^i(x),u^l(z)\}\right)\ud x\\
=\binom{L}{T}\int \frac{\dev g(z)}{\dev u^l_L}e^{\mlambda\cdot(x-z)}\left(\mdev_z^TP^{li}_S(z)\right)\mdev_z^{L-T+S}\delta(x-z)\ud x
\end{multline*}
where the second line is obtained by simply expanding the derivations of the bracket. By substituting as usual $\mdev_z\delta(x-z)$ with $(-\mdev_x)\delta(x-z)$ and integrating by parts we get
\begin{multline*}
 \binom{L}{T} \frac{\dev g(z)}{\dev u^l_L}\mlambda^{L-T+S}\left(\mdev_z^TP^{li}_S(z)\right)\\
=\frac{\dev g(z)}{\dev u^l_L}(\mlambda+\mdev_z)^TP^{li}_S(z)\mlambda^S\\
=\{u^i_{\mlambda}g\}
\end{multline*}
where the last equality is given by \eqref{eq:MasterRight}. Summarizing, we have
\begin{equation}\label{eq:JacobiFour1}
\int e^{\mlambda\cdot(x-z)}e^{\mmu\cdot(y-z)}\{u^i(x),\{u^j(y),u^k(z)\}\}\ud x\ud y=\{u^i_{\mlambda}\{u^j_{\mmu}u^k\}\}. 
\end{equation}
The second term for the Jacobi identity among three coordinate functions is the same with $u^i(x)$ replaced by $u^j(y)$. The same computations hold provided the switching, and this gives as second term of the Fourier transform of the Jacobi identity
\begin{equation}\label{eq:JacobiFour2}
\int e^{\mlambda\cdot(x-z)}e^{\mmu\cdot(y-z)}\{u^j(y),\{u^i(x),u^k(z)\}\}\ud x\ud y=\{u^j_{\mmu}\{u^j_{\mlambda}u^k\}\}.
\end{equation}
The RHS term of the PVA-Jacobi identity is more complicated to achieve. As before, let us start from expanding the usual formula for the Poisson bracket
\begin{multline*}
\int e^{\mlambda\cdot(x-z)}e^{\mmu\cdot(y-z)}\{\{u^i(x),u^j(y)\},u^k(z)\}\\
=\int e^{\mlambda\cdot(x-z)}e^{\mmu\cdot(y-z)}\frac{\dev \{u^i(x),u^j(y)\}}{\dev u^l_L(y)}\mdev^L_y\{u^l(y),u^k(z)\}\ud x\ud y\\
=\int e^{\mlambda\cdot(x-z)}e^{\mmu\cdot(y-z)}\frac{\dev \{u^i(x),u^j(y)\}}{\dev u^l_L(y)}\mdev^L_y\left(P^{kl}_M(z)\mdev_z^M\delta(y-z)\right)\ud x\ud y\\
=\int e^{\mlambda\cdot(x-z)}e^{\mmu\cdot(y-z)}\frac{\dev \{u^i(x),u^j(y)\}}{\dev u^l_L(y)}P^{kl}_M(z)\mdev^L_y\mdev_z^M\delta(y-z)\ud x\ud y
\end{multline*}
The derivative with respect to $y$ in the third does not act on $P^{kl}_M(z)$, so we could move it further. Moreover, for convenience we can trade $\mdev^L_y\mdev^M_z\delta(y-z)$ for $(-1)^{|L|}(-\mdev_y)^{M+L}\delta(y-z)$ exchanging two times the variables respect to which we derive the Dirac's delta. It allows us to integrate by parts the delta's derivatives, in order to get
\begin{multline*}
=(-1)^{|L|}\int \mdev_{y}^{L+M}\left(e^{\mmu\cdot(y-z)}\frac{\dev \{u^i(x),u^j(y)\}}{\dev u^l_L(y)}\right)e^{\mlambda\cdot(x-z)}P^{kl}_M(z)\delta(y-z)\ud x\ud y\\
=(-1)^{|L|}\binom{L+M}{T}\int\mdev_y^T\left(\frac{\dev \{u^i(x),u^j(y)\}}{\dev u^l_L(y)}\right)\mmu^{L+M-T} e^{\mmu\cdot(y-z)}e^{\mlambda\cdot(x-z)}\cdot\\\cdot P^{kl}_M(z)\delta(y-z)\ud x\ud y\\
=(-1)^{|L|}\binom{L+M}{T}\int\mdev_z^T\left(\frac{\dev \{u^i(x),u^j(z)\}}{\dev u^l_L(z)}\right)\mmu^{L+M-T} e^{\mlambda\cdot(x-z)}P^{kl}_M(z)\ud x.
 \end{multline*}
From the form for $\{u^i(x),u^j(z)\}$ we see that the partial derivatives act only on the coefficients $P^{ji}_N$. So, we get that our expression is equal to
\begin{multline*}
 (-1)^{|L|}\binom{L+M}{T}\mmu^{L+M-T}\cdot\\\cdot\int\mdev_z^T\left(\frac{\dev P^{ji}_N(z)}{\dev u^l_L(z)}\mdev^N_z\delta(x-z)\right) e^{\mlambda\cdot(x-z)} P^{kl}_M(z)\ud x
\end{multline*}
Basically we repeat the computation applying the same rules for multiderivatives of product and the integration by parts of the Dirac's delta and we end, after the integration, with
\begin{equation*}
(-1)^{|L|}\binom{L+M}{T}\binom{T}{R}P^{kl}_M(z)\mmu^{L+M-T}\mlambda^{T-R+N}\mdev_z^R\frac{\dev P^{ji}_N(z)}{\dev u^l_L(z)}. 
\end{equation*}
The rules for the product of binomials hold also in the multiindices case, since they are only a product of ordinary binomials. It means, by slightly abusing the notation, that
\begin{equation*}
\binom{A}{B}\binom{B}{C}=\frac{A!}{B!(A-B)!}\frac{B!}{C!(B-C)!}=\binom{A}{A-B,C} 
\end{equation*}
In our case, calling $L+M-T=Q$, we get
\begin{multline*}
 (-1)^{|L|}\binom{L+M}{Q,R}P^{kl}_M(z)\mmu^Q\mlambda^{L+M-Q-R}\mdev^R\frac{\dev P^{ji}_N(z)}{\dev u^l_L(z)} \mlambda^N\\
=(-1)^{|L|}P^{kl}_M(z)(\mlambda+\mmu+\mdev)^{L+M}\frac{\dev P^{ji}_N(z)}{\dev u^l_L(z)} \mlambda^N.
\end{multline*}
Finally, we can adsorb the sign in front of the expression, and we get
\begin{equation*}
 P^{kl}_M(z)(\mlambda+\mmu+\mdev)^M(-\mlambda-\mmu-\mdev)^L\frac{\dev P^{ji}_N(z)}{\dev u^l_L(z)} \mlambda^N
\end{equation*}
which is clearly the expression in terms of \eqref{eq:MasterLeft} of $\{u^i_{\mlambda} u^j\}_{\mlambda+\mmu}u^k\}$, the RHS of the PVA-Jacobi identity. We have finally proved that taking the double Fourier transform with respect to $e^{\mlambda\cdot(x-z)+\mmu\cdot(y-z)}$ of the Jacobi identity for the Poisson bracket of the generators gives the PVA-Jacobi identity among them.
\section{Components of the 1st order deformation of the $\lambda$-bracket for $d=2$}\label{app:DefoPar}
\subsection{Deformation of \eqref{eq:defP1}}\label{app:DefoPar1}
\begin{align*}
B^{1,11}_{11}&=0&B^{2,22}_{22}&=0\\
B^{2,22}_{11}&=0&B^{1,11}_{22}&=0\\
B^{2,21}_{11}&=0&B^{1,12}_{22}&=0\\
\tilde{B}^{1,11}_{12}&=0&\tilde{B}^{2,22}_{12}&=0\\
B^{1,21}_{11}&=-B^{2,11}_{11}&B^{2,12}_{22}&=-B^{1,22}_{22}\\
\tilde{B}^{1,12}_{12}&=-B^{2,11}_{22}+\frac{\dev \tilde{A}^{11}}{\dev q}&\tilde{B}^{2,21}_{12}&=-B^{1,22}_{11}-\frac{\dev \tilde{A}^{22}}{\dev p}\\
\tilde{B}^{2,11}_{12}&=-\frac{1}{2}B^{1,12}_{11}-\frac{\dev \tilde{A}^{12}}{\dev p}&\tilde{B}^{1,22}_{12}&=-\frac{1}{2}B^{2,21}_{22}+\frac{\dev \tilde{A}^{12}}{\dev q}\\
\tilde{B}^{1,21}_{12}&=\frac{\dev \tilde{A}^{12}}{\dev p}&\tilde{B}^{2,12}_{12}&=-\frac{\dev \tilde{A}^{12}}{\dev q}\\
B^{2,12}_{11}&=B^{1,22}_{11}+2\frac{\dev \tilde{A}^{22}}{\dev q}&B^{1,21}_{22}&=B^{2,11}_{22}-2\frac{\dev \tilde{A}^{11}}{\dev p}
\end{align*}
\begin{align*}
 \tilde{D}^{11,1}&=0&\tilde{D}^{22,2}&=0\\
\tilde{D}^{12,1}&=-\frac{1}{8}B^{1,12}_{11}&\tilde{D}^{12,2}&=\frac{1}{8}B^{2,21}_{22}\\
\tilde{D}^{22,1}&=\frac{1}{2}\left(B^{2,11}_{22}-\frac{\dev\tilde{A}^{11}}{\dev q}\right)&\tilde{D}^{22,2}&=-\frac{1}{2}\left(B^{1,22}_{11}+\frac{\dev\tilde{A}^{22}}{\dev q}\right)
\end{align*}
\begin{align*}
 \tilde{C}^{11,11}&=0&\tilde{C}^{22,22}&=0\\
\tilde{C}^{11,12}&=\frac{1}{4}\left(\frac{\dev B^{2,11}_{22}}{\dev p}-\frac{\dev^2\tilde{A}^{11}}{\dev p\dev q}\right)&\tilde{C}^{21,22}&=-\frac{1}{4}\left(\frac{\dev B^{1,22}_{11}}{\dev q	}+\frac{\dev^2\tilde{A}^{22}}{\dev p\dev q}\right)\\
\tilde{C}^{11,21}&=-\frac{1}{8}\frac{\dev B^{1,12}_{11}}{\dev p}&\tilde{C}^{12,22}&=\frac{1}{8}\frac{\dev B^{2,21}_{22}}{\dev p}\\
\tilde{C}^{12,12}&=\frac{1}{2}\left(\frac{\dev B^{2,11}_{22}}{\dev q}-\frac{\dev^2\tilde{A}^{11}}{\dev q^2}\right)&\tilde{C}^{21,21}&=-\frac{1}{2}\left(\frac{\dev B^{1,22}_{11}}{\dev p}+\frac{\dev^2\tilde{A}^{22}}{\dev p^2}\right)\\
\tilde{C}^{12,21}&=0&\tilde{C}^{11,22}&=\frac{1}{8}\left(\frac{\dev B^{2,21}_{22}}{\dev p}-\frac{\dev B^{2,12}_{11}}{\dev q}\right)
\end{align*}
\subsection{Deformation of \eqref{eq:defP2}}\label{app:DefoPar2}
\begin{align*}
B^{1,11}_{22}&=0&B^{2,21}_{11}&=0\\
B^{1,12}_{11}&=0&B^{2,22}_{11}&=0\\
B^{2,22}_{22}&=0&B^{1,21}_{22}&=-\left(B^{1,12}_{22}+B^{2,11}_{22}\right)\\
B^{1,22}_{11}&=-\frac{\dev\tilde{A}^{22}}{\dev q}&B^{2,11}_{11}&=\frac{\dev\tilde{A}^{22}}{\dev p}\\
B^{2,12}_{11}&=\frac{\dev\tilde{A}^{22}}{\dev q}&B^{1,22}_{22}&-\left(B^{2,12}_{22}+B^{2,21}_{22}\right)\\
\tilde{B}^{2,21}_{11}&=0&\tilde{B}^{2,22}_{12}&=0\\
\tilde{B}^{1,12}_{12}&=-\frac{1}{2}B^{1,11}_{11}&\tilde{B}^{1,11}_{12}&=-\frac{1}{2}B^{1,12}_{22}\\
\tilde{B}^{1,22}_{12}&=-B^{2,11}_{11}+\frac{\dev\tilde{A}^{12}}{\dev q}&\tilde{B}^{2,21}_{12}&=\frac{1}{2}\left(B^{1,21}_{11}+\frac{\dev\tilde{A}^{22}}{\dev p}\right)\\
\tilde{B}^{2,11}_{12}&=B^{2,12}_{22}+\frac{1}{2}B^{2,21}_{22}-\frac{\dev\tilde{A}^{11}}{\dev q}+\frac{\dev\tilde{A}^{12}}{\dev p}&\tilde{B}^{2,12}_{12}&=B^{2,11}_{22}-\frac{1}{2}B^{1,21}_{11}-\frac{\dev\tilde{A}^{12}}{\dev q}-\frac{1}{2}\frac{\dev\tilde{A}^{22}}{\dev p}
\end{align*}
\begin{align*}
 \tilde{D}^{11,1}&=\frac{1}{4}B^{1,12}_{22}&\tilde{D}^{11,2}&=-\frac{1}{4}B^{1,11}_{11}\\
\tilde{D}^{12,1}&=\frac{1}{8}\left(B^{1,11}_{11}-B^{2,21}_{22}\right)&\tilde{D}^{12,2}&=-\frac{1}{8}\left(B^{1,21}_{11}+\frac{\dev\tilde{A}^{22}}{\dev p}\right)\\
\tilde{D}^{22,1}&=-\frac{1}{2}\left(B^{1,21}_{11}+\frac{\dev\tilde A^{22}}{\dev p}\right)&\tilde D^{22,2}&=0
\end{align*}
\begin{align*}
 \tilde C^{11,11}&=\frac{1}{4}\frac{\dev B^{1,12}_{22}}{\dev p}&\tilde C^{12,22}&=-\frac{1}{8}\left(\frac{\dev B^{1,21}_{11}}{\dev q}+\frac{\dev^2\tilde A^{22}}{\dev p\dev q}\right)\\
\tilde C^{21,22}&=\frac{1}{8}\left(\frac{\dev B^{1,21}_{11}}{\dev q}+\frac{\dev^2\tilde A^{22}}{\dev p\dev q}\right)&\tilde C^{22,22}&=0\\
\tilde C^{11,21}&=\frac{1}{8}\left(\frac{\dev B^{1,11}_{11}}{\dev p}-\frac{\dev B^{2,21}_{22}}{\dev p}\right)&\tilde C^{11,12}&=\frac{1}{8}\left(\frac{\dev B^{1,12}_{22}}{\dev q}-\frac{\dev B^{1,11}_{11}}{\dev p}\right)\\
\tilde C^{12,12}&=-\frac{1}{4}\frac{\dev B^{1,11}_{11}}{\dev q}&\tilde C^{21,21}&=\frac{1}{4}\left(\frac{\dev B^{1,21}_{11}}{\dev p}+\frac{\dev^2\tilde A^{22}}{\dev p^2}\right)
\end{align*}
\begin{align*}
\tilde C^{11,22}&=\frac{1}{8}\left(2\frac{\dev B^{1,11}_{11}}{\dev q}-2\frac{\dev B^{2,12}_{22}}{\dev q}-\frac{\dev B^{2,21}_{22}}{\dev q}-2\frac{\dev B^{1,21}_{11}}{\dev p}+2\frac{\dev^2\tilde A^{11}}{\dev q^2}-4\frac{\dev^2\tilde A^{12}}{\dev p\dev q}\right)\\
\tilde C^{12,21}&=-\frac{1}{8}\left(2\frac{\dev B^{2,12}_{22}}{\dev q}+\frac{\dev B^{1,11}_{11}}{\dev q}-\frac{\dev B^{1,21}_{11}}{\dev p}+2\frac{\dev^2\tilde A^{22}}{\dev p^2}+2\frac{\dev^2\tilde A^{11}}{\dev q^2}-4\frac{\dev^2\tilde A^{12}}{\dev p\dev q}\right)
\end{align*}

\subsection{Deformation of \eqref{eq:LPPVA}}\label{app:DefoParLP}
\begin{align*}
B^{1,11}_{11}&=\frac{q^2}{p^2}\left(\frac{\dev \tilde{A}^{22}}{\dev q}+2 B^{1,22}_{11}\right)+\frac{q}{p}\left(-2 \frac{\dev \tilde{A}^{12}}{\dev q}+\frac{5}{2} \frac{\dev \tilde{A}^{22}}{\dev p}+2 B^{1,21}_{11}+B^{2,11}_{11}\right)+\\
&-\left(-\frac{\dev \tilde{A}^{11}}{\dev q}+\frac{\dev \tilde{A}^{12}}{\dev p}+B^{1,22}_{22}+B^{2,12}_{22}\right)+\frac{p \left(\frac{\dev \tilde{A}^{11}}{\dev p}-2 B^{2,11}_{22}\right)}{2 q}+\\
&+\frac{\tilde{A}^{12}}{p}-\frac{2 q \tilde{A}^{22}}{p^2}-\frac{\tilde{A}^{11}}{q}\\
B^{2,22}_{22}&=\frac{p^2}{q^2}\left(-\frac{\dev\tilde{A}^{11}}{\dev p}+2 B^{2,11}_{22}\right)+\frac{p}{q} \left(-\frac{5}{2} \frac{\dev\tilde{A}^{11}}{\dev q}+2 \frac{\dev \tilde{A}^{12}}{\dev p}+B^{1,22}_{22}+2 B^{2,12}_{22}\right)+\\
&-\left(\frac{\dev\tilde{A}^{22}}{\dev p}-\frac{\dev\tilde{A}^{12}}{\dev q}+B^{1,21}_{11}+B^{2,11}_{11}\right)+\frac{q \left(-\frac{\dev\tilde{A}^{22}}{\dev q}-2 B^{1,22}_{11}\right)}{2 p}+\\
&+\frac{2 p\tilde{A}^{11}}{q^2}-\frac{\tilde{A}^{12}}{q}+\frac{\tilde{A}^{22}}{p}
\end{align*}
\begin{align*}
B^{1,11}_{22}&=B^{2,22}_{11}=0\\
B^{1,12}_{11}&=-\frac{2 p}{q}B^{1,11}_{11}\\
B^{2,21}_{22}&=-\frac{2 q}{p} B^{2,22}_{22}\\
B^{2,12}_{11}&=2 \frac{p^2}{q^2}B^{1,11}_{11}-\frac{p}{q}\left(2 B^{1,21}_{11}+2 B^{2,11}_{11}\right)-B^{1,22}_{11}\\
B^{1,21}_{22}&=2 \frac{q^2}{p^2}B^{2,22}_{22}-\frac{q}{p}\left(2 B^{1,22}_{22}+2 B^{2,12}_{22}\right)-B^{2,11}_{22}\\
B^{2,21}_{11}&=\frac{p}{q}\left(B^{1,21}_{11}+B^{2,11}_{11}\right)-\frac{p^2}{q^2}B^{1,11}_{11}\\
B^{1,12}_{22}&=\frac{q}{p}\left(B^{1,22}_{22}+B^{2,12}_{22}\right)-\frac{q^2}{p^2}B^{2,22}_{22}\\
\tilde{B}^{1,11}_{12}&=-\frac{1}{2} \left(\frac{q}{p}\left(B^{1,22}_{22}+B^{2,12}_{22}\right)-\frac{q^2}{p^2} B^{2,22}_{22}\right)\\
\tilde{B}^{2,22}_{12}&=-\frac{1}{2} \left(\frac{p}{q}(B^{1,21}_{11}+B^{2,11}_{11})-\frac{p^2}{q^2} B^{1,11}_{11}\right)
\end{align*}
\begin{align*}
\tilde{B}^{1,12}_{12}&=\frac{1}{2} \left(-\frac{2q}{p}B^{2,22}_{22}-B^{1,11}_{11}+2 B^{1,22}_{22}+2 B^{2,12}_{22}\right)\\
\tilde{B}^{2,21}_{12}&=\frac{1}{2} \left(-\frac{2 p}{q}B^{1,11}_{11}+2 B^{1,21}_{11}+2 B^{2,11}_{11}-B^{2,22}_{22}\right)\\
\tilde{B}^{1,22}_{12}&=-\frac{\dev \tilde{A}^{12}}{\dev q}+2\frac{\dev \tilde{A}^{22}}{\dev p}+ \frac{q}{p} \frac{\dev \tilde{A}^{22}}{\dev q}-\frac{2}{p}\tilde{A}^{22}+(2 B^{1,21}_{11}+B^{2,11}_{11})-\frac{3 p}{2q} B^{1,11}_{11}+\frac{2q}{p} B^{1,22}_{11}\\
\tilde{B}^{2,11}_{12}&=-2\frac{\dev \tilde{A}^{11}}{\dev q}+\frac{\dev \tilde{A}^{12}}{\dev p}-\frac{p}{q}\frac{\dev \tilde{A}^{11}}{\dev p}+\frac{2}{q}\tilde{A}^{11}+(B^{1,22}_{22}+2 B^{2,12}_{22})+\frac{2p}{q} B^{2,11}_{22}-\frac{3 q}{2p} B^{2,22}_{22}\\
\tilde{B}^{1,21}_{12}&=\frac{1}{2} \left(-\frac{4 \tilde{A}^{11}}{q}+4 \frac{\dev \tilde{A}^{11}}{\dev q}-2 \frac{\dev \tilde{A}^{12}}{\dev p}+\frac{2 p \frac{\dev \tilde{A}^{11}}{\dev p}}{q}-\frac{4 p B^{2,11}_{22}}{q}+\frac{3 q B^{2,22}_{22}}{p}+2 B^{1,11}_{11}-3 B^{1,22}_{22}-5 B^{2,12}_{22}\right)\\
\tilde{B}^{2,12}_{12}&=\frac{1}{2} \left(\frac{4 \tilde{A}^{22}}{p}+2 \frac{\dev \tilde{A}^{12}}{\dev q}-4 \frac{\dev \tilde{A}^{22}}{\dev p}-\frac{2 q \frac{\dev \tilde{A}^{22}}{\dev q}}{p}+\frac{3 p B^{1,11}_{11}}{q}-\frac{4 q B^{1,22}_{11}}{p}-5 B^{1,21}_{11}-3 B^{2,11}_{11}+2 B^{2,22}_{22}\right)
\end{align*}
\begin{align*}
\tilde{C}^{11,11}_{12}&=\frac{1}{4} \frac{\dev }{\dev p}\left(\frac{q (B^{1,22}_{22}+B^{2,12}_{22})}{p}-\left(\frac{q}{p}\right)^2 B^{2,22}_{22}\right)\\
\tilde{C}^{22,22}_{12}&=-\frac{1}{4} \frac{\dev }{\dev q}\left(\frac{p (B^{1,21}_{11}+B^{2,11}_{11})}{q}-\left(\frac{p}{q}\right)^2 B^{1,11}_{11}\right)\\
\tilde{C}^{11,12}_{12}&=-\frac{1}{8} \left(\frac{\dev }{\dev q}\left(\left(\frac{q}{p}\right)^2 B^{2,22}_{22}-\frac{q (B^{1,22}_{22}+B^{2,12}_{22})}{p}\right)+\frac{\dev }{\dev p}\left(-\frac{2 q B^{2,22}_{22}}{p}+B^{1,11}_{11}+2 B^{1,22}_{22}+2 B^{2,12}_{22}\right)\right)\\
\tilde{C}^{21,22}_{12}&=\frac{1}{8} \left(\frac{\dev }{\dev q}\left(-\frac{2 p B^{1,11}_{11}}{q}+2 B^{1,21}_{11}+2 B^{2,11}_{11}+B^{2,22}_{22}\right)+\frac{\dev }{\dev p}\left(\left(\frac{p}{q}\right)^2 B^{1,11}_{11}-\frac{p (B^{1,21}_{11}+B^{2,11}_{11})}{q}\right)\right)\\
\tilde{C}^{12,12}_{12}&=-\frac{1}{4} \frac{\dev }{\dev q}\left(-\frac{2 q B^{2,22}_{22}}{p}+B^{1,11}_{11}+2 B^{1,22}_{22}+2 B^{2,12}_{22}\right)\\
\tilde{C}^{21,21}_{12}&=\frac{1}{4} \frac{\dev }{\dev p}\left(-\frac{2 p B^{1,11}_{11}}{q}+2 B^{1,21}_{11}+2 B^{2,11}_{11}+B^{2,22}_{22}\right)\\
\tilde{C}^{12,22}_{12}&=-\frac{1}{8} \frac{\dev (B^{1,21}_{11}+B^{2,11}_{11}+2 B^{2,22}_{22})}{\dev q}\\
\tilde{C}^{11,21}_{12}&=\frac{1}{8} \frac{\dev (2 B^{1,11}_{11}+B^{1,22}_{22}+B^{2,12}_{22})}{\dev p}\\
\tilde{C}^{11,22}_{12}&=\frac{1}{4} \left(2 \left(\frac{\dev ^2\tilde{A}^{11}}{\dev q^2}-\frac{\dev ^2\tilde{A}^{22}}{\dev p^2}\right)-\frac{\dev }{\dev p}\left(\frac{q \left(\frac{\dev \tilde{A}^{22}}{\dev q}+2 B^{1,22}_{11}\right)}{p}-\frac{2 \tilde{A}^{22}}{p}-\frac{3 p B^{1,11}_{11}}{2 q}+3 B^{1,21}_{11}+2 B^{2,11}_{11}\right)\right.\\
&+\left.\frac{\dev }{\dev q}\left(\frac{p \left(\frac{\dev \tilde{A}^{11}}{\dev p}-2B^{2,11}_{22}\right)}{q}-\frac{2 \tilde{A}^{11}}{q}+\frac{3 q B^{2,22}_{22}}{2 p}+2 B^{1,11}_{11}-B^{1,22}_{22}-2 B^{2,12}_{22}\right)+\frac{p }{q}\frac{\dev B^{1,11}_{11}}{\dev p}\right)\\
\tilde{C}^{12,21}_{12}&=\frac{1}{8} \left(4 \frac{\dev ^2\tilde{A}^{11}}{\dev q^2}-4 \frac{\dev ^2\tilde{A}^{22}}{\dev p^2}+\right.\\
&+\left.\frac{\dev }{\dev p}\left(-\frac{2 q \left(\frac{\dev \tilde{A}^{22}}{\dev q}+2 B^{1,22}_{11}\right)}{p}+\frac{4 \tilde{A}^{22}}{p}+\frac{3 p B^{1,11}_{11}}{q}-5 B^{1,21}_{11}-3 B^{2,11}_{11}+2 B^{2,22}_{22}\right)\right.\\
&\left.+\frac{\dev }{\dev q}\left(\frac{2 p \left(\frac{\dev \tilde{A}^{11}}{\dev p}-2B^{2,11}_{22}\right)}{q}-\frac{4 \tilde{A}^{11}}{q}+\frac{3 q B^{2,22}_{22}}{p}+2 B^{1,11}_{11}-3 B^{1,22}_{22}-5 B^{2,12}_{22}\right)\right.\\
&\left.+\frac{2 q \frac{\dev B^{2,22}_{22}}{\dev q}}{p}\right)
\end{align*}
\begin{align*}
\tilde{D}^{11,1}_{12}&=\frac{1}{4} \left(\frac{q (B^{1,22}_{22}+B^{2,12}_{22})}{q}-\frac{q^2 B^{2,22}_{22}}{q^2}\right)\\
\tilde{D}^{22,2}_{12}&=-\frac{1}{4} \left(\frac{q (B^{1,21}_{11}+B^{2,11}_{11})}{q}-\frac{q^2 B^{1,11}_{11}}{q^2}\right)\\
\tilde{D}^{11,2}_{12}&=-\frac{1}{4} \left(-\frac{2 q B^{2,22}_{22}}{q}+B^{1,11}_{11}+2 B^{1,22}_{22}+2 B^{2,12}_{22}\right)\\
\tilde{D}^{22,1}_{12}&=\frac{1}{4} \left(-\frac{2 q B^{1,11}_{11}}{q}+2 B^{1,21}_{11}+2 B^{2,11}_{11}+B^{2,22}_{22}\right)\\
\tilde{D}^{12,1}_{12}&=\frac{1}{8} (2 B^{1,11}_{11}+B^{1,22}_{22}+B^{2,12}_{22})\\
\tilde{D}^{12,2}_{12}&=-\frac{1}{8} (B^{1,21}_{11}+B^{2,11}_{11}+2 B^{2,22}_{22})
\end{align*}

\section*{Acknowledgements}
I want to thank in particular prof.~B.~Dubrovin for introducing me to the subject of the deformation of Poisson brackets of hydrodynamic type, and for the accurate supervision of my research. I'm very grateful for the discussions with D.~Valeri, A.~De Sole and F.~Magri about Poisson Vertex Algebras during the the series of seminars in Febraury 2013 at Universit\`a degli Studi di Milano--Bicocca and with P.~Lorenzoni for his useful advice about multidimensional Poisson brackets. I wish to express heartfelt thanks to the anonimous referee for suggesting important improvements both in the structure and in the content of this paper. 

This work is partially supported by the European Research Council Advanced Grant FroM--PDE and by PRIN 2010--11 Grant “Geometric and analytic 
theory of Hamiltonian systems in finite and infinite dimensions” of Italian Ministry of Universities and Researches.
\bibliography{biblio}

\end{document}